\documentclass{article}
\usepackage{iclr2027_conference,times}
\iclrfinalcopy
\usepackage{graphicx} 
\usepackage{xcolor}
\usepackage{booktabs, siunitx}
\usepackage{bm}
\newcommand{\sci}[2]{\ensuremath{#1\!\times\!10^{#2}}}
\newcommand{\scib}[2]{\ensuremath{\bm{#1\!\times\!10^{#2}}}}
\usepackage[colorlinks=true, citecolor=blue, linkcolor=black, urlcolor=black]{hyperref}
\usepackage{amsmath, amsthm, amssymb}
\usepackage{stmaryrd}
\usepackage{hyperref}
\usepackage{enumitem}
\usepackage{comment}
\usepackage{algorithm}
\usepackage{algpseudocode}
\usepackage{wrapfig}
\usepackage{lipsum}

\newcommand{\R}{\mathbb{R}}
\newcommand{\real}{\R}

\newcommand{\trasp}[1]{{#1}^{\!\top}}

\newcommand{\mcN}{\mathcal{N}}

\newcommand{\norm}[1]{\left\|#1\right\|}

\newcommand{\relu}{\mathrm{ReLU}}

\newcommand{\until}[1]{\{1,\dots, #1\}}

\newcommand{\setdef}[2]{\{#1 \; | \; #2\}}

\newcommand{\map}[3]{#1: #2 \rightarrow #3}

\newtheorem{theorem}{Theorem}
\newtheorem{lemma}[theorem]{Lemma}
\newtheorem{proposition}[theorem]{Proposition}

\newcommand{\jac}[1]{J_{#1}}
\newcommand{\distrib}{\mathcal{D}}
\newcommand{\Aref}[1]{\hyperref[assm:A#1]{\textup{(A#1)}}}
\newcommand{\Bref}[1]{\hyperref[assm:B#1]{\textup{(B#1)}}} 

\title{Constrained Flow Matching \\ via Lagrangian Dual Flows}
\date{April 2026}
\author{Vince Kurtz \\ DePaul University \\ \texttt{vkurtz1@depaul.edu}
\And Alexander Davydov \\ Rice University \\ \texttt{davydov@rice.edu}}

\begin{document}

\maketitle

\begin{abstract}
    Flow matching is a powerful tool for generative modeling, but emerging applications in robotics, planning, and control require inference-time constraints on generated outputs. Such constraints are often complex and highly nonlinear. As a result, methods designed for linear constraints like image inpainting are rarely sufficient, and projection or optimization-based alternatives can be prohibitively expensive. In this paper, we introduce Lagrangian Dual Flows, a new family of constrained generation techniques based on Lagrangian dual dynamics. By flowing a dual co-state alongside generated samples, we can guarantee nonlinear constraint satisfaction without expensive optimization subproblems, pseudoinverses, or projection steps during the denoising process. The resulting constrained generation algorithms are simple, effective, and open new theoretical connections between flow matching and primal-dual methods in numerical optimization.
\end{abstract}

\section{Introduction}

Flow matching~\citep{lipman2023flow} has emerged as an effective and widely-used approach for generative modeling, with applications in image and video generation \citep{lipman2023flow, ho2022video, jin2025pyramidal}, robotics \citep{black2024pi_0, zhang2025flowpolicy, kurtz2025generative}, physics \citep{utkarsh2025pcfm,baldan2026physics}, and beyond. Flow matching learns a vector field that is integrated at inference time to transport samples from a reference distribution (e.g., standard Gaussian) to an approximation of the data distribution. 

However, many generative modeling applications require the generated output to satisfy constraints. In robotics, for example, generated actions should avoid collisions while generated trajectories must satisfy dynamic feasibility constraints. In these emerging application areas, constraints are critical but (i) are only known at inference time and (ii) are often nonlinear. To meet these needs, we aim to take a pre-trained flow matching model and modify the generation procedure such that the generated samples satisfy constraints. 

Existing methods for doing so either (i) require simple linear constraints~\citep{song2023pigdm,pokle2024training}, or (ii) solve a computationally expensive subproblem at each integration step in the denoising process~\citep{li2025hardflow,utkarsh2025pcfm}. Linear constraints are a special case of linear inverse problems, a well-studied problem class that includes image inpainting and super-resolution, but these techniques are not directly applicable to general nonlinear constraints. These and related methods often rely on computing a pseudoinverse of the constraint, which can be difficult to compute or ill-defined. Methods that support nonlinear constraints require projection to the constraint manifold~\citep{utkarsh2025pcfm} or the solution of optimization subproblems~\citep{li2025hardflow,safeflow2025} in each inference step, both of which can be computationally demanding.

Note that while linear methods are typically derived as approximate Bayesian posterior sampling, steering samples toward a conditional distribution given the constraints \citep{song2023pigdm}, nonlinear constrained methods typically focus only on constraint satisfaction. We follow the latter group, and ask only that generated samples be \emph{feasible} with respect to the constraints.

\begin{figure}
    \centering
    \includegraphics[width=0.95\linewidth]{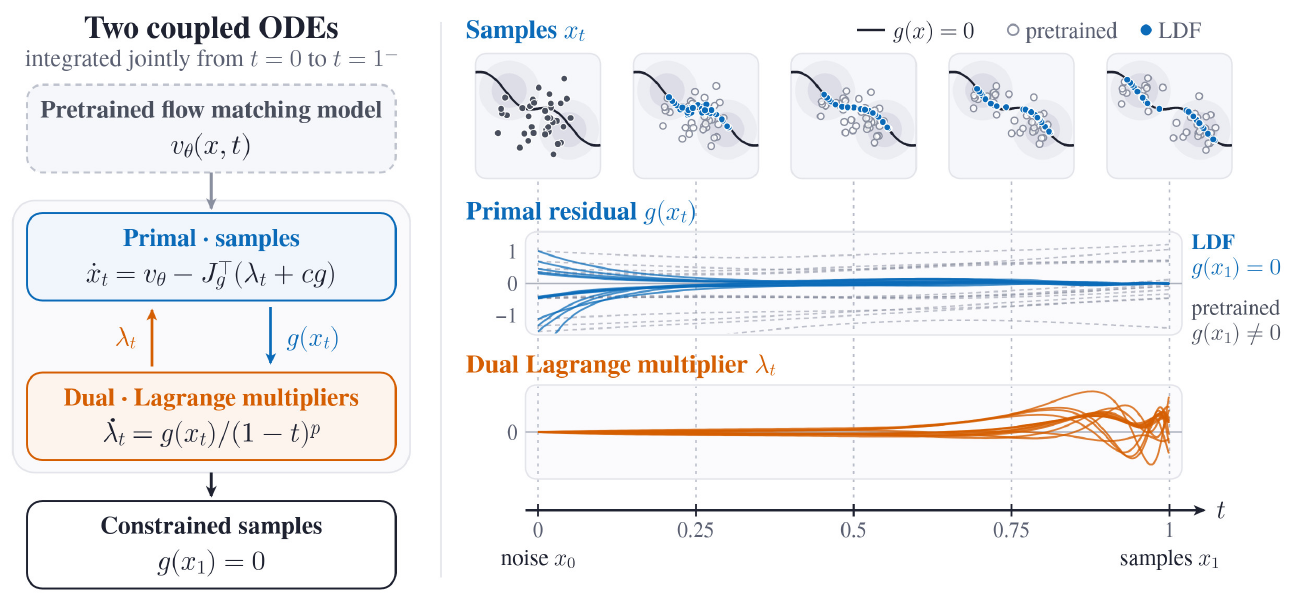}
    \caption{Visual summary of LDF, which integrates two coupled ODEs. Lagrange multipliers $\lambda_t$ flow alongside the sample $x_t$, applying an adaptive correction that guarantees constraint satisfaction to tight numerical tolerances. LDF is simple, efficient, supports nonlinear equality and inequality constraints, and requires no specialized tooling beyond the off-the-shelf ODE solver already used for unconstrained generation.}
    \vspace{-1.5em}
    \label{fig:hero-figure}
\end{figure}

To this end, we present a family of efficient methods for constraint satisfaction in flow matching, which we call \emph{Lagrangian Dual Flows} (LDF). Inspired by the classical view of constrained optimization as a continuous-time dynamical system~\citep{arrow1958studies,platt1987constrained}, we augment the state of the flow matching model with a dual variable that plays the role of a Lagrange multiplier. The dual dynamics are a simple ordinary differential equation (ODE) defined by the constraint; the same off-the-shelf ODE integration tools typically used for flow matching can thus generate samples satisfying constraints. LDF handles both nonlinear equality and inequality constraints and does not require any pseudoinverse or computationally expensive optimization sub-problems. 

Beyond these structural advantages, LDF comes with provable guarantees. We show that the constraint violation is driven to zero at an explicit rate $O\big((1-t)^\alpha\big)$ as $t \to 1^-$ for an explicit constant $\alpha > 0$. 
This result holds under standard regularity and constraint-qualification assumptions, together with a small-variation condition requiring that the Gram matrix associated with the constraint Jacobian stay sufficiently close to some fixed reference matrix along the trajectory, which holds automatically for affine constraints.
We emphasize that this is a guarantee of \emph{feasibility} of the generated sample, not of convergence to any particular conditional distribution.

In summary, our contributions are as follows:
\vspace{-0.5em}
\begin{itemize}
  \item We introduce \emph{Lagrangian Dual Flows}, a family of methods for enforcing inference-time
  constraints on a pre-trained flow matching model by
  augmenting the generative ODE with dual (and, for inequality constraints, slack)
  variables. Sampling requires only the same off-the-shelf ODE solvers used for
  unconstrained flow matching, enabling simple and computationally efficient implementations. \vspace{-0.3em}

  \item Our proposed approach supports general nonlinear equality and inequality
  constraints known only at inference time. Inequalities are handled via a
  slack reformulation whose dynamics keep the slack
  nonnegative automatically (Lemma~\ref{lem:posinv}).\vspace{-0.3em}

  \item We prove that the constraint violation converges to zero as $t \to 1^-$
  with an explicit algebraic rate (Theorems~\ref{thm:p2}
  and~\ref{thm:ineq}). \vspace{-0.3em}

  \item We demonstrate the effectiveness of our approach on a variety of challenging constrained generation problems, including nonlinear equality constraints, image inpainting, collision-free path planning, and legged robot locomotion. \vspace{-0.5em}
\end{itemize}

The remainder of this paper is organized as follows: Section~\ref{sec:related_work} reviews existing methods for constrained flow matching, while Section~\ref{sec:background} presents background and a formal problem setting. Section~\ref{sec:ldf} presents our main results, including algorithms for equality and inequality-constrained generation and theoretical convergence analysis. We illustrate the effectiveness of LDF with several examples in Section~\ref{sec:numerical_results}, and conclude with Section~\ref{sec:conclusion}. 

\section{Related Work}\label{sec:related_work}

Inference-time constrained generation with flow matching is an active area of research. We summarize the main families of approaches here.

\textbf{Penalty and guidance with quadratic costs.} 
Some training-free methods add the gradient of a quadratic constraint-violation penalty to the generative dynamics~\citep{ho2022video,chung2022mcg,feng2025guidance}. This guidance is inherently soft. Tight constraint satisfaction requires a larger penalty which, in turn, stiffens the ODE. Inspired by primal-dual methods in numerical optimization, LDF retains a penalty term but augments it with a dual variable that drives the residual to zero exactly while limiting numerical ill-conditioning.

\textbf{Pseudoinverse and posterior sampling guidance.} 
A second family approximates the conditional score with a Gaussian posterior, resulting in corrections built from the pseudoinverse of the constraint Jacobian~\citep{song2023pigdm,chung2023dps,peng2024improving,pokle2024training,pourya2025flower,black2026real}. These methods are designed for posterior sampling under (approximately) linear-Gaussian measurement models, and the per-step cost scales with the number of constraints. LDF includes a dual correction term that recovers a similar structure when specialized to linear constraints, but replaces the pseudoinverse with a dual variable
governed by a simple ODE.

\textbf{Projection and reflection-based methods.}
A third family enforces constraints geometrically by projecting iterates onto the feasible set or by reflecting the dynamics at the boundary of this set~\citep{cheng2024eci,deng2024reflectedsb,xie2024reflected,utkarsh2025pcfm,ma2026polyflow}. These methods are effective for simple constraints when projection is cheap, but exact projection onto a general nonlinear constraint set is itself an optimization problem, and hard projection inside an ODE yields discontinuous dynamics~\citep{nagurney2012projected}. Our soft-projected slack dynamics (Section~\ref{sec:inequality}) keep the dynamics Lipschitz while retaining convergence guarantees.

\textbf{Control-theoretic methods.}
A fourth family borrows machinery from trajectory optimization, model predictive control, and control barrier functions (CBFs), enforcing constraints through an optimization subproblem or safety filter at each integration step~\citep{li2025hardflow,xiao2023safediffuser,safeflow2025,yang2025uniconflow}. These methods ensure tight constraint satisfaction, but the per-step optimization subproblem can be computationally expensive. Additionally, CBF-based methods require finding a barrier function compatible with the dynamics~\citep{ames2019control}. This can pose a considerable challenge when the dynamics are defined by a large pre-trained flow matching model. 

\textbf{ODE-based optimization.}
Finally, our approach draws on the classical view of optimization as a dynamical system, beginning with the basic differential multiplier method of Platt and Barr~\citep{platt1987constrained} and the broader literature on continuous-time primal-dual and saddle-point flows~\citep{arrow1958studies,feijer2010stability,cherukuri2017saddle,qu2019exponential}. LDF couples the dual dynamics to the generative ODE and integrates both with the same off-the-shelf solver such that feasibility is achieved at the end of the denoising process.

\section{Background}\label{sec:background}

\textit{Generative modeling and flow matching.} The goal of generative modeling is to draw new samples from an unknown data distribution $\distrib$ on $\R^n$ given access only to a finite collection of samples $x^{(i)} \sim \distrib$, $i \in \{1,\dots,N\}$. Flow matching~\citep{lipman2023flow} approaches this challenge by learning a time-varying vector field $\map{v_\theta}{\R^n \times [0,1]}{\R^n}$ that transports a simple reference distribution into $\distrib$, where $\theta$ denotes trainable parameters. Typically, a sample is generated by drawing an initial state $x_0 \sim \mathcal{N}(0, I)$ from a standard normal distribution and integrating the ODE
\begin{equation}\label{eq:flow}
    \dot{x}_t = v_\theta(x_t,t)
\end{equation}
from $t = 0$ to $t = 1$. The endpoint $x_1$ is then a sample approximately distributed according to $\distrib$. The vector field $v_\theta$ is trained so that its flow matches a prescribed probability path connecting the Gaussian reference to the data distribution. We treat $v_\theta$ as frozen and refer the reader to~\citep{lipman2023flow} for details on how to train such a model. A key feature of flow matching is that sampling requires integrating~\eqref{eq:flow}, typically with off-the-shelf ODE solvers.

\textit{Constraints.} In many applications, the generated sample $x_1$ is required to satisfy constraints. We consider equality and inequality constraints,
\begin{equation}\label{eq:constraints}
    g(x_1) = 0, \qquad h(x_1) \leq 0,
\end{equation}
where $\map{g}{\R^n}{\R^m}$ and $\map{h}{\real^n}{\real^{k}}$. We understand inequality constraints componentwise, and the constraints are permitted to be nonlinear.

Crucially, we assume that $g$ and $h$ are only known at inference time. A robot may encounter a new obstacle, an inpainting mask may be specified by a user, or a physical conservation law may be imposed only at deployment. The vector field $v_\theta$ is thus trained without knowledge of these constraints and $\distrib$ may be supported outside of the feasible set. As a result, an unconstrained sample $x_1$ may violate the constraints~\eqref{eq:constraints}.

\textit{Objective.} Our goal is to take a pre-trained $v_\theta$ and, at inference time, modify the sampling process such that the generated sample satisfies the constraints. We seek a modification that satisfies three properties. First, it should certify constraint satisfaction: the generated $x_1$ should provably satisfy~\eqref{eq:constraints} at least up to small numerical tolerances. Second, it should accommodate nonlinear $g$ and $h$. Third, it should remain computationally lightweight, ideally requiring no more than the integration of an augmented ODE with the same off-the-shelf solvers used for unconstrained sampling. In particular, we would like to avoid solving optimization subproblems (e.g., projection onto the constraint set) or expensive nonlinear pseudoinverse computations at each integration step.

\section{Lagrangian Dual Flows}\label{sec:ldf}

In this section we introduce LDF, our family of methods for enforcing constraints on a pre-trained flow matching model at inference time. 
Before presenting our approach, we motivate its form by examining two classical approaches to enforcing an equality constraint $g(x)=0$ on the output of a flow~\eqref{eq:flow}. Both approaches admit natural formulations within the flow matching framework, and each has a shortcoming that LDF addresses.

\textit{Penalty method.} A simple way to drive the sample toward the constraint set is to add the negative gradient of a quadratic penalty $\frac{c}{2}\|g(x)\|^2$ to the flow:
\begin{equation}
\dot x_t = v_\theta(x_t, t) - c\,\trasp{\jac{g}(x_t)}\,g(x_t),
\label{eq:penalty}
\end{equation}
where $c > 0$ is a constant and $\jac{g}(x) = \nabla_x g(x)$ is the Jacobian of $g$. This correction pulls trajectories toward the constraint manifold $\setdef{x \in \R^n}{g(x) = 0}$ but exact constraint satisfaction requires $c \to \infty$. For any finite $c > 0$, the penalty must balance the pre-trained $v_\theta$ and a residual persists at $t = 1$. At the same time, large $c$ makes the ODE stiff and computationally expensive to integrate.

\textit{Tangential projection.} A more principled approach is to require that the corrected flow remain tangent to the constraint manifold, so that trajectories never move away from $g(x) = 0$ after reaching it. Augmenting~\eqref{eq:flow} with a correction of the form $-\trasp{\jac{g}(x)}\,\mu$ and
imposing $\tfrac{d}{dt} g(x_t) = 0$ determines $\mu(x,t) = \big(\jac{g}\trasp{\jac{g}}\big)^{-1}\jac{g}\,v_\theta(x,t)$.
Substituting back and retaining a penalty term yields
\begin{equation}\label{eq:projected-flow}
    \dot x_t = \big(I - \jac{g}(x_t)^\dagger \,\jac{g}(x_t)\big)\,v_\theta(x_t, t) - c\trasp{\jac{g}(x_t)} g(x_t),
\end{equation}
where $\jac{g}(x)^\dagger = \trasp{\jac{g}(x)}\,(\jac{g}(x)\,\trasp{\jac{g}(x)})^{-1}$ is the Moore-Penrose pseudoinverse of $\jac{g}(x)$ and the penalty term ensures that $g$ converges to zero. Note that we can choose a smaller $c > 0$ in this case, as $\mu(x, t)$ eliminates movement away from the constraint manifold. 
The cost is that every evaluation of the
right-hand side requires solving an $m \times m$ system with
$\jac{g}\trasp{\jac{g}}$, which becomes expensive for complex and high-dimensional constraints, and unstable when that matrix is ill-conditioned.

\subsection{The Proposed Method}\label{sec:method-dynamics}
Let $\map{g}{\real^n}{\real^m}$ be the equality constraint in~\eqref{eq:constraints} and let $\jac{g}(x) \in \real^{m \times n}$ denote its Jacobian matrix evaluated at $x \in \real^n$. Inspired by constrained differential optimization~\citep{platt1987constrained}, we augment the flow matching state $x$ with a dual variable $\lambda \in \real^m$, and evolve the pair $(x,\lambda)$ jointly. For $t \in [0,1)$ and fixed constants $c > 0, p \in [1,\infty)$, the \emph{Lagrangian dual flow} is the system
\begin{subequations}\label{eq:ldf}
\begin{align}
\dot x_t &= v_\theta(x_t, t) - \trasp{ \jac{g}(x_t)}\,\lambda_t - c\,\trasp{\jac{g}(x_t)}\,g(x_t), \label{eq:x}\\
\dot \lambda_t &= \frac{g(x_t)}{(1-t)^p}, \label{eq:lambda}
\end{align}
\end{subequations}
with initial conditions $x_0 \sim \mcN(0, I_n)$, $\lambda_0 = 0$. A constrained sample is generated by integrating~\eqref{eq:ldf} from $t = 0$ to $t = 1$ and returning $x_1$, as outlined in Algorithm~\ref{alg:equality_constrained}.

\begin{algorithm}
\caption{Equality-constrained Lagrangian Dual Flows}\label{alg:equality_constrained}
\begin{algorithmic}[1]
\Require Trained flow model $v_\theta(x, t)$, constraint $g(x)$, penalty $c > 0$, constant $p \geq 1$.
\Ensure Final state $x_1$ such that $g(x_1) = 0$.
\State Initialize the state $x_0 \sim \mathcal{N}(0, I_n)$.
\State Initialize the dual $\lambda_0 = 0$.
\State Integrate from $t=0$ to $t=1$:
\begin{align*}
    \dot x_t &= v_\theta(x_t, t) - \trasp{ \jac{g}(x_t)}\,\lambda_t - c\,\trasp{\jac{g}(x_t)}\,g(x_t), \\
    \dot \lambda_t &= g(x_t) / (1-t)^p.
\end{align*}
\State (Optional) perform $N_{\mathrm{proj}}$ Gauss-Newton iterations to project $x_1$ to the constraint manifold.

\Return $x_1$
\end{algorithmic}
\end{algorithm}

The dynamics~\eqref{eq:ldf} modify the unconstrained flow~\eqref{eq:flow} through two correction terms in~\eqref{eq:x} and the introduction of an auxiliary equation~\eqref{eq:lambda}. First, the term $-c \jac{g}(x)^\top g(x)$ in~\eqref{eq:x} pulls the sample toward the constraint set with the strength of the pull controlled by the weight $c > 0$, as in a quadratic penalty method. Second, the term $-\jac{g}(x)^\top \lambda$ applies a correction in directions normal to the constraint set with magnitude set by the dual variable $\lambda$. Rather than a fixed penalty or a quantity computed from $x$, $\lambda$ is a system state and is integrated over time according to the dynamics~\eqref{eq:lambda}. The time-rescaling factor $(1-t)^{-p}$ makes this accumulation increasingly aggressive as $t \to 1^-$, concentrating the correction near the end of the integration interval. For numerical stability, our implementation uses $(1-t+\epsilon)^{-p}$, with $\epsilon = \sci{1}{-6}$. This rescaling is necessary because, unlike in \citep{platt1987constrained} where integration is from $t=0$ to $t=\infty$, flow models integrate from $t=0$ to $t=1$. The exponent $p > 0$ governs the rate of this concentration. We take $p=2$ throughout and justify this choice in Section~\ref{sec:theory}. Intuitively, growth in $(1-t)^{-p}$ is balanced by decay in $g(x_t)$.

At the end of Algorithm~\ref{alg:equality_constrained}, optional Gauss-Newton iterations project the generated sample onto the set $\{x \mid g(x) = 0\}$. We adapt this step from \citep{utkarsh2025pcfm}. It is computationally cheap, since $v_\theta(x, t)$ is not used, and allows us to achieve constraint satisfaction to tight numerical tolerances. In practice, we find that this projection step does not change the quantitative nature of generated samples, and we limit ourselves to $N_{\mathrm{proj}} = 2$ projection iterations.

\textit{Computational cost.} Sampling from~\eqref{eq:ldf} only requires the integration of an augmented ODE, carried out with the same off-the-shelf integrators used for unconstrained flow matching. Compared to unconstrained flow matching, LDF adds $m$ scalar ODE channels for the dual variables and, per evaluation of the right-hand-side, a single vector-Jacobian product (VJP) $\trasp{[\lambda_t + c g(x_t)]}J_g(x_t)$. This VJP is cheaply available via automatic differentiation. 

\subsection{Theoretical Analysis}
\label{sec:theory}

We now establish that LDF drives the equality constraint violation to zero as $t \to 1^-$. We discuss inequality constraints in Section~\ref{sec:inequality}. Our analysis rests on the following standing assumptions, which we take to hold along the trajectory of the dynamics~\eqref{eq:ldf}.

\begin{itemize}
    
\item[(A1)]\phantomsection\label{assm:A1} \emph{Smoothness:}
$g \in C^2(\R^n; \R^m)$ and $v_\theta \in C^1(\R^n \times [0,1); \R^n)$, and there exist constants $L, B \geq 0$ such that $\norm{\jac{g}(x)}_{2} \leq L$ and $\norm{v_\theta(x,t)} \le B$ along the trajectory.

\item[(A2)]\phantomsection\label{assm:A2} \emph{Constraint qualification:}
There exist constants $0 < a_{\min} \leq a_{\max} < \infty$ such that $a_{\min} I_m \preceq \jac{g}(x)\trasp{\jac{g}(x)} \preceq a_{\max} I_m$ along the trajectory.

\item[(A3)]\phantomsection\label{assm:A3} \emph{Bounded trajectory:}
The trajectory $(x_t, \lambda_t)$ exists on $[0,1)$ with $x_t$ lying in a compact set.
\end{itemize}

Assumption~\Aref{1} is a standard regularity condition. Assumption~\Aref{2} is a uniform full-rank constraint qualification condition on the constraint Jacobian, ensuring that the constraint directions remain well-conditioned along the trajectory. Assumption~\Aref{3} rules out finite-time blowup.

\textit{Main result ($p=2$).} Before stating the theorem, we describe the change of variables underlying the analysis. The singularity of~\eqref{eq:lambda} at $t = 1$ is naturally handled by passing to the time variable $\tau := -\log(1-t)$, which maps $[0,1)$ onto $[0, \infty)$, together with the rescaled dual variable $\eta := (1-t)\lambda$. In these variables convergence as $t \to 1^-$ corresponds to convergence as $\tau \to \infty$. The analysis then proceeds by constructing a Lyapunov function for the pair $(g, \eta)$.
To ensure that the Lyapunov function decreases along solution trajectories,
the construction additionally requires a condition controlling the variation of the constraint Jacobian:

\begin{itemize}                 \item[(A4)]\phantomsection\label{assm:A4} \emph{Small variation:} There exists a reference matrix $A_* \succ 0$ such that, along the trajectory, $\rho := K_{A_*} \sup_t \norm{\jac{g}(x_t)\trasp{\jac{g}(x_t)} - A_*}_{2} < 1$, where $K_{A_*} = \norm{A_*^{-1} + \tfrac{1}{2}I}_{2} + 1$.     
\end{itemize}     

We note that Assumption~\Aref{4} holds automatically with $\rho = 0$ for affine constraints, where $\jac{g}(x)\trasp{\jac{g}(x)}$ is constant.
Let $A_*$ and $\rho < 1$ be the reference matrix and small-variation constant
from~\Aref{4}, let
$P := \left(\begin{smallmatrix} 2A_*^{-1}+I_m & -I_m \\ -I_m & I_m+A_*
\end{smallmatrix}\right) \succ 0$, and set
$\alpha := (1-\rho)/\lambda_{\max}(P) > 0$, where $\lambda_{\max}(P)$ is the maximum eigenvalue of $P$. We then obtain the following result:

\begin{theorem}[Convergence at $p = 2$]\label{thm:p2}
Under Assumptions~\Aref{1}--\Aref{4} with $p = 2$, the trajectory
of~\eqref{eq:x}--\eqref{eq:lambda} satisfies, as $t \to 1^-$,
\begin{equation*}
    \norm{g(x_t)} = O\big((1-t)^{\alpha}\big)
    \quad \text{and} \quad
    (1-t)\norm{\lambda_t} = O\big((1-t)^{\alpha}\big).
\end{equation*}
In particular, since $\alpha > 0$, both quantities vanish as $t \to 1^-$ and the
generated sample is feasible.
\end{theorem}

In particular, for affine constraints the conclusions of Theorem~\ref{thm:p2} hold without Assumption~\Aref{4}.

\textit{The role of $p$.} The exponent $p$ in~\eqref{eq:lambda} controls how aggressively $\lambda$ accumulates as $t \to 1^-$. For $p \in [1,2)$, we can prove (see the Appendix, Proposition~\ref{prop:small-p}) that $\|g(x_t)\|$ converges as $t \to 1^-$ but that its limiting value is generally nonzero. For $p > 2$ we have no general proof of convergence, although a scalar example converges analytically and numerical experiments suggest convergence more broadly. On the other hand, the dynamics become increasingly challenging to integrate as $p$ increases. We therefore suggest $p = 2$ as a reasonable tradeoff between constraint satisfaction and ease of integration. We include more thorough theoretical analysis of different $p$ in Appendix~\ref{app:different-p}.

\subsection{Inequality Constraints}
\label{sec:inequality}

We now extend the method to inequality constraints $h(x) \le 0$. 
We convert the inequality $h(x) \leq 0$ into an equality by introducing a slack variable $s \in \R^k$ and writing
\begin{equation*}
    h(x) + s = 0, \qquad s \geq 0,
\end{equation*}
so that $h(x) + s = 0$ together with $s \geq 0$ recovers $h(x) \leq 0$. Treating $\tilde g(x, s) := h(x) + s$ as an equality constraint on the augmented state $(x, s)$ and applying the flow of Section~\ref{sec:method-dynamics} would yield dynamics for $(x, s, \lambda)$ analogous to~\eqref{eq:ldf}. The difficulty is the nonnegativity requirement $s \geq 0$: the slack dynamics produced this way do not keep $s$ nonnegative on their own, and enforcing $s \geq 0$ by hard projection in the ODE produces a projected dynamical system~\citep{nagurney2012projected}, which is discontinuous.

\textit{Soft projection.} We can instead enforce $s \geq 0$ through a Lipschitz relaxation. The resulting Lagrangian dual flow for inequality constraints is
\begin{subequations}
\begin{align}
\dot x_t &= v_\theta(x_t, t) - \trasp{\jac{h}(x_t)}\,\lambda_t - c\,\trasp{\jac{h}(x_t)}\,(h(x_t) + s_t), \label{eq:ineq-x}\\
\dot s_t &= c\left[-s_t + \mathrm{ReLU}\!\left(-h(x_t) - \frac{1}{c}\lambda_t\right)\right], \label{eq:ineq-s}\\
\dot \lambda_t &= \frac{h(x_t) + s_t}{(1-t)^p}, \label{eq:ineq-lambda}
\end{align}
\end{subequations}
with $x_0 \sim \mathcal{N}(0, I_n)$, $\lambda_0 = 0$, $s_0 = \relu(-h(x_0))$, and $\mathrm{ReLU}(y) = \max(y, 0)$. The slack equation~\eqref{eq:ineq-s} is a globally Lipschitz relaxation of the projected dynamics. 
This approach is inspired by work on continuous-time projected and proximal gradient dynamics~\citep{XingBaoGao2003,HASSANMOGHADDAM2021109311,davydov2025time}.
Moreover, these dynamics ensure nonnegativity of the slack variable.

\begin{lemma}[Positive invariance of the slack variable]
\label{lem:posinv}
Along the slack dynamics~\eqref{eq:ineq-s} with $s_0 \geq 0$, the slack satisfies $s_t \geq 0$ for all $t \in [0,1)$.
\end{lemma}

\textit{Convergence.} Under similar assumptions as in the equality-constrained case, but now imposed on $h$ in place of $g$ (see Assumptions~\Bref{1}-\Bref{4} in Appendix~\ref{app:proofs} for specifics), we establish the soft-projected dynamics drive the generated sample to satisfy the inequality constraints as $t \to 1^-$.

Let $A_*$, $\kappa_0$, and $\rho < \kappa_0$ be the reference matrix, dissipation
margin, and small-variation constant from~\Bref{4}, let $P \succ 0$ be the matrix
defining the Lyapunov function in Lemma~\ref{lem:clf}, and set
$\alpha := (\kappa_0 - \rho)/\lambda_{\max}(P) > 0$.

\begin{theorem}[Convergence for inequality constraints at $p = 2$]
\label{thm:ineq}
Suppose Assumptions~\Bref{1}--\Bref{4} in Appendix~\ref{app:proofs} hold. Then with $p = 2$ the
soft-projected dynamics~\eqref{eq:ineq-x}--\eqref{eq:ineq-lambda} satisfy, as
$t \to 1^-$,
\begin{equation*}
    \norm{\relu\big(h(x_t)\big)} = O\big((1-t)^{\alpha}\big)
    \quad \text{and} \quad
    (1-t)\norm{\lambda_t} = O\big((1-t)^{\alpha}\big).
\end{equation*}
In particular, since $\alpha > 0$, both quantities vanish as $t \to 1^-$ and the generated sample is feasible.
\end{theorem}

For affine constraints,~\Bref{4} holds automatically and
the conclusions follow from~\Bref{1}--\Bref{3} alone.

\textit{Reduced dynamics.} Equation~\eqref{eq:ineq-s} has $s$ tend toward $s^\star(x,\lambda) = \relu(-h(x)-\lambda/c)$. For large $c > 0$, this slack variable converges quickly and we may replace $s_t$ by $s^\star(x_t,\lambda_t)$. This substitution gives the reduced dynamics
\begin{subequations}\label{eq:ineq-reduced}
\begin{align}
\dot{x}_t &= v_\theta(x_t,t) - \jac{h}(x_t)^\top \, \relu(\lambda_t + c\, h(x_t)) \\ 
\dot{\lambda}_t &= \frac{\relu(\lambda_t + c\, h(x_t)) - \lambda_t}{c\,(1-t)^p}
\end{align}
\end{subequations}
These dynamics are compatible with Algorithm~\ref{alg:equality_constrained}, and are the variant we use in our experiments. It removes the slack channels and we find it integrates at least as accurately as the full system~\eqref{eq:ineq-x}--\eqref{eq:ineq-lambda} at a lower cost. We leave a comprehensive convergence analysis of these dynamics to future work.

\section{Numerical Results}\label{sec:numerical_results}

We implemented LDF in JAX, using FLAX/NNX for neural networks and Diffrax~\citep{kidger2021on} for ODE integration. Experiments were run on a small laptop GPU (NVIDIA RTX Pro 2000, 8GB). Code at \url{https://anonymous.4open.science/r/constrained_flow_matching}.

We compare against several constrained generation algorithms. The \textbf{penalty} baseline adds a simple quadratic penalty on constraint violation, integrating \eqref{eq:penalty} during the denoising process. The \textbf{PCFM} (physics-constrained flow matching) baseline \citep{utkarsh2025pcfm} solves an optimization subproblem at each denoising step to project samples to the constraint manifold. \textbf{$\mathbf{\Pi}$GDM} extends the approach of \citep{song2023pigdm, pokle2024training} to handle nonlinear constraints via iterative linearization. The \textbf{CBF} baseline enforces constraints by solving a (relaxed) CBF quadratic program at each denoising step, as described in \citep{safeflow2025}. We test against two versions of our proposed approach: \textbf{LDF} without optional projection steps, and \textbf{LDF + projection} with $N_{\mathrm{proj}} = 2$ Gauss-Newton projection iterations following the denoising process.

Note that PCFM and ${\Pi}$GDM support equality constraints only, while CBF supports inequality constraints only. Our proposed LDF methods support both types of constraints.

We test these algorithms on six constrained generation problems of increasing complexity. 

\begin{figure}
    \centering
    \includegraphics[width=0.9\linewidth]{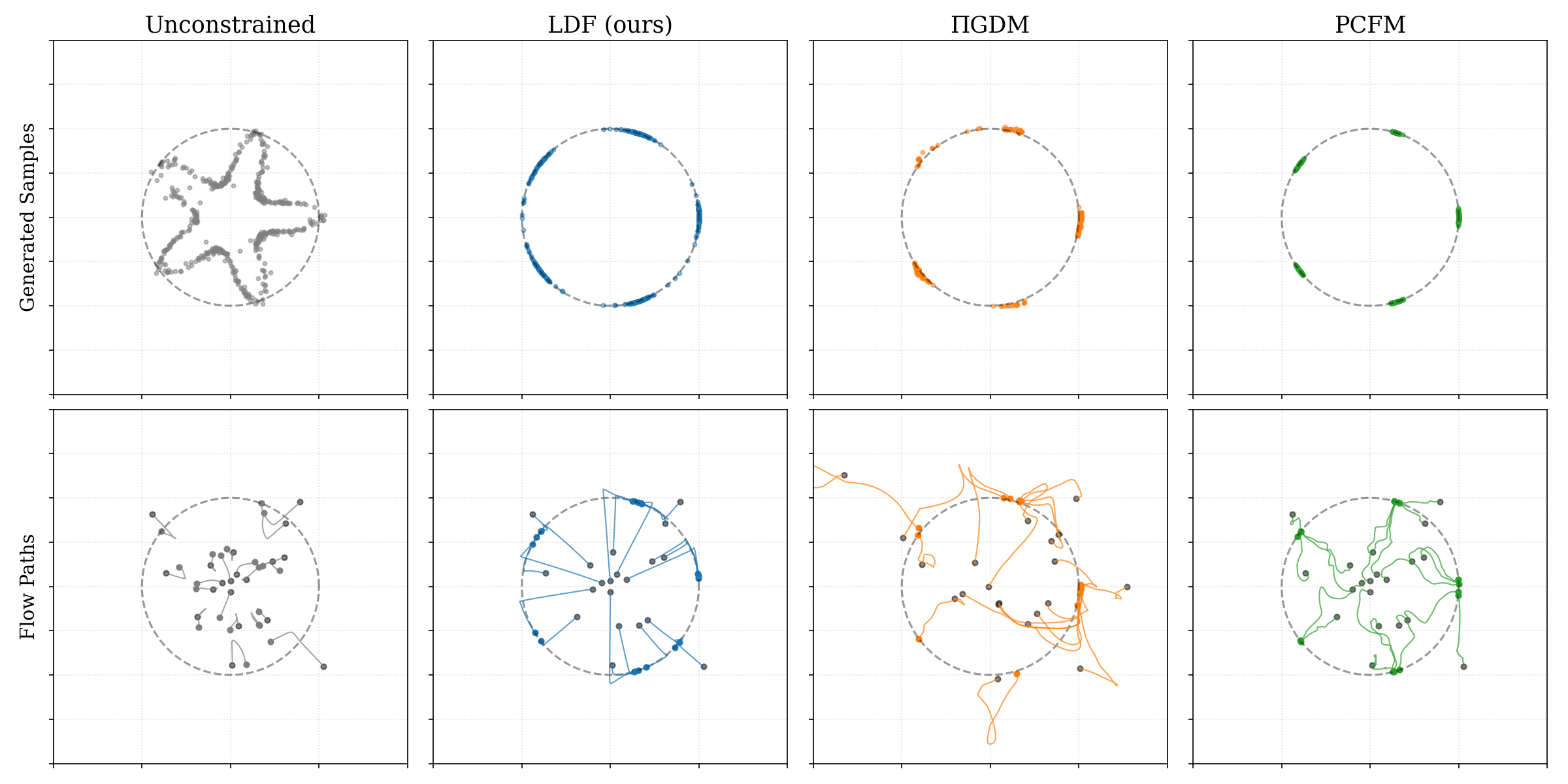}
    \vspace{-0.65em}
    \caption{
    Generated star examples subject to a unit circle constraint. Our LDF method produces clean paths that drive samples directly toward the constraint manifold, unlike the PCFM \citep{utkarsh2025pcfm} and $\Pi$GDM \citep{song2023pigdm, pokle2024training} baselines. LDF and PCFM both achieve tight constraint satisfaction. $\Pi$GDM, as a method designed for linear constraints, produces samples slightly off the nonlinear constraint manifold. \vspace{-1.7em}}
    \label{fig:constrained_star}
\end{figure}

In the \textbf{Star (eq.)} example, a simple flow model generates samples from a star-shaped submanifold of $\mathbb{R}^2$, as shown in Fig.~\ref{fig:constrained_star}. Samples are constrained to lie on the unit circle. We also consider an inequality constrained version of the same problem, \textbf{Star (ineq.)}, where samples are constrained to the right half of the plane (Fig.~\ref{fig:inequality_constrained_examples}a). The star examples provide a simple, low-dimensional illustration of the method's ability to handle different constraint types, as well as an opportunity to visualize the generation process. Fig.~\ref{fig:constrained_star} shows that LDF produces clean generation paths that drive directly toward the constraint manifold, while $\Pi$GDM and PCFM take more circuitous generation paths.

Increasing the problem dimension, we come to \textbf{MNIST} image inpainting. Here an unconditional flow model is trained on MNIST~\citep{lecun1998mnist} to generate arbitrary digits. The top half of the image is fixed to match a reference from the MNIST test set, as shown in Fig.~\ref{fig:constrained_mnist}. This simple linear constraint is a special case of the nonlinear equality constraints that LDF can handle. Nonetheless, we find that LDF provides performance competitive with inpainting-specific algorithms like $\Pi$GDM.

\begin{wrapfigure}{l}{0.5\linewidth}
    \centering
    \includegraphics[width=\linewidth,trim={0 0 0 0.2cm},clip]{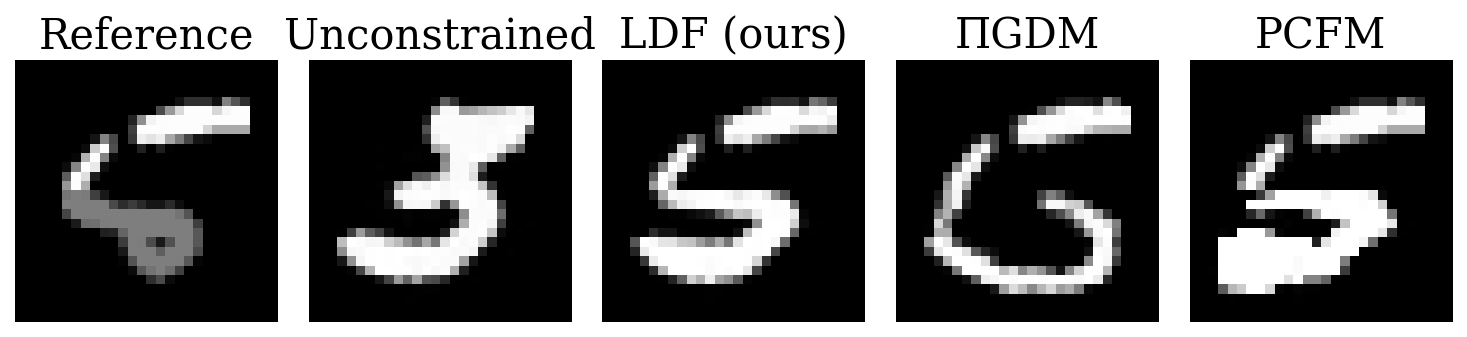}
    \vspace{-1.15em}
    \caption{A pre-trained model generates MNIST handwritten digits. We then constrain the top half of the image to match the reference (left). LDF produces clean digits consistent with the inpainting mask. PCFM \citep{utkarsh2025pcfm} achieves tight constraint satisfaction, but at the cost of visible artifacts in the generated image. Additional samples are shown in Appendix~\ref{subsec:MNIST}.}
    \vspace{-0.5em}
    \label{fig:constrained_mnist}
\end{wrapfigure}

The \textbf{Obstacle} avoidance examples generate control points of a Bezier spline representing a robot's path through a field of circular obstacles, as shown in Fig.~\ref{fig:inequality_constrained_examples}b. The constraint is that the path must avoid collisions with obstacles, as evaluated along 90 evenly-spaced points along the path. This is an example of a highly non-convex inequality constraint. The CBF baseline can also produce collision-free paths, but is much slower than LDF. Additionally, the CBF-QP is by default infeasible and must be relaxed with an elastic formulation as suggested in \citep{safeflow2025}. This relaxation eliminates any constraint satisfaction guarantees and some samples as a result do collide with the obstacles. We also find that the CBF method occasionally diverges, resulting in NaN errors.

Finally, we consider the \textbf{Hopper} and \textbf{Walker} examples described in \citep{xiao2023safediffuser, yang2026safeflowmatcher}. In these examples, a flow matching network is trained on expert data to generate trajectories (state-action pairs over time) for planar robots. The constraint is a velocity-dependent ceiling on the robot's height, as illustrated in Fig.~\ref{fig:inequality_constrained_examples}c. LDF-generated trajectories preserve the general shape of the phase portrait (torso height vs vertical velocity) while also enforcing the constraint.

Table~\ref{tab:generation_times} shows average generation time and constraint violation across 20 samples from each scenario. The two LDF variants consistently produce tight constraint satisfaction and fast generation times. This is thanks in large part to the simplicity of the method: an LDF implementation adds only a few lines of code and a single vector-Jacobian product; no expensive pseudoinverses or optimization subproblems are required during the denoising process. 

\begin{figure}
    \centering
    \includegraphics[width=0.8\linewidth]{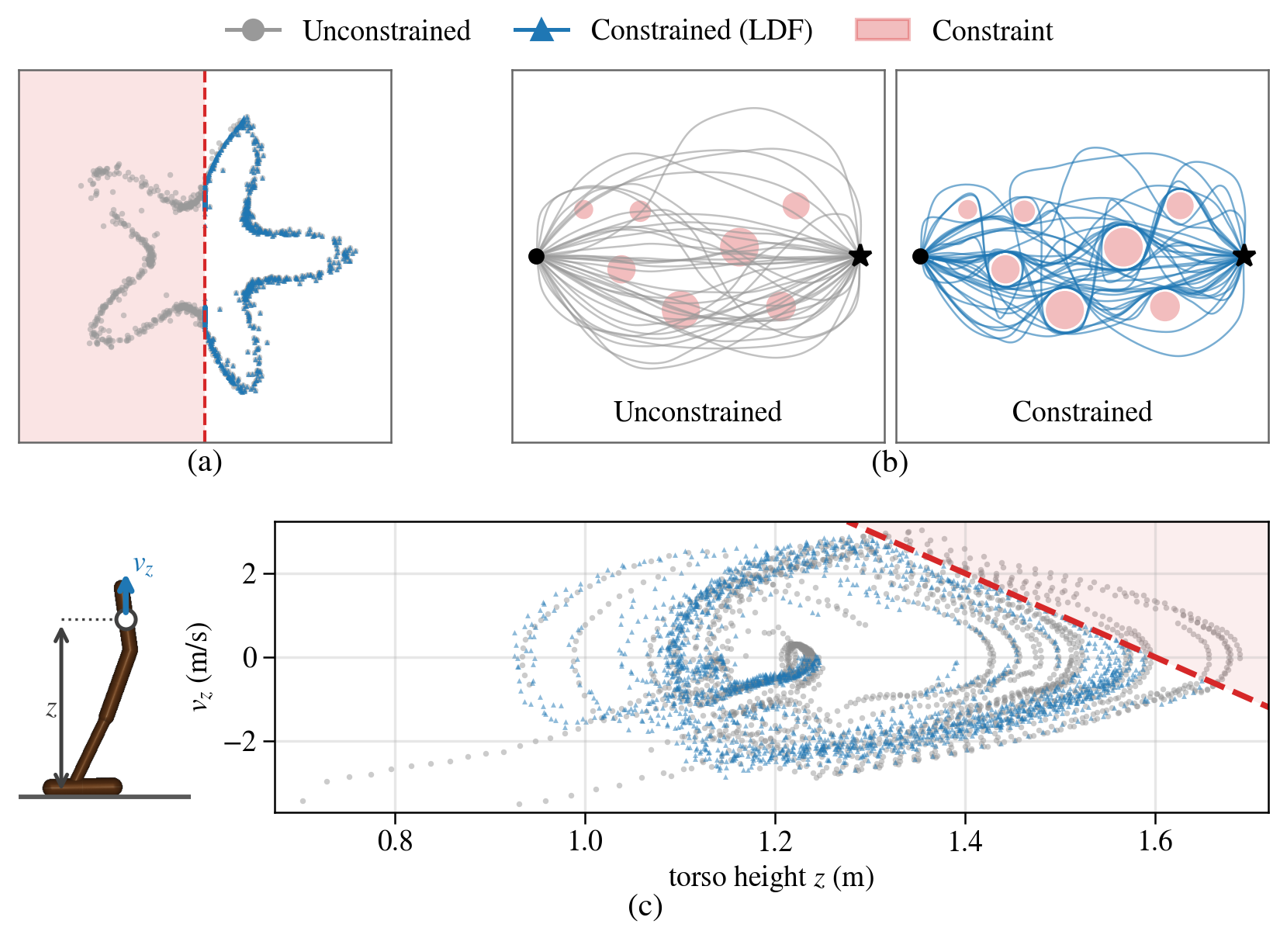}
    \vspace{-1.5em}
    \caption{Examples of inequality-constrained test problems. (a) Inequality-constrained star; LDF preserves the shape of the data manifold while producing only samples that fall on the right half of the plane. (b) Obstacle avoidance; LDF avoids obstacles while providing smooth paths from a start to a goal. (c) Hopper; LDF generates trajectories for a hopping robot (left) that enforce a velocity dependent ceiling, as shown in the phase portrait on the right.}
    \label{fig:inequality_constrained_examples}
\end{figure}

\begin{table*}[t]
\centering
\resizebox{\textwidth}{!}{%
\begin{tabular}{lcccccccccccc}
\toprule
& \multicolumn{2}{c}{Penalty only}
& \multicolumn{2}{c}{CBF
}
& \multicolumn{2}{c}{PCFM
}
& \multicolumn{2}{c}{$\Pi$GDM
}
& \multicolumn{2}{c}{LDF (no proj.)}
& \multicolumn{2}{c}{LDF + proj.} \\
\cmidrule(lr){2-3}\cmidrule(lr){4-5}\cmidrule(lr){6-7}\cmidrule(lr){8-9}\cmidrule(lr){10-11}\cmidrule(lr){12-13}
Problem & Time & Viol. & Time & Viol. & Time & Viol. & Time & Viol. & Time & Viol. & Time & Viol. \\
\midrule
Star (eq.) & \textbf{7.9} & \sci{3.6}{-2} & --- & --- & 42.9 & \scib{2.4}{-8} & 15.3 & \sci{3.1}{-2} & \textbf{9.2} & \sci{2.6}{-3} & \textbf{8.4} & \scib{3.3}{-8} \\
MNIST & \textbf{182.0} & \sci{3.3}{-1} & --- & --- & 224.6 & \scib{1.0}{-7} & 543.4 & \sci{1.1}{-1} & \textbf{182.3} & \sci{5.0}{-2} & \textbf{185.9} & \scib{6.9}{-8} \\
\midrule
Star (ineq.) & \textbf{7.6} & \sci{2.2}{-2} & 331.5 & \textbf{0} & --- & --- & --- & --- & \textbf{8.7} & \sci{9.3}{-3} & \textbf{8.0} & \textbf{0} \\
Obstacles (2) & \textbf{48.4} & \sci{1.0}{-3} & 2365.1 & \textbf{0}~$\dagger$ & --- & --- & --- & --- & \textbf{45.2} & \sci{7.3}{-5} & \textbf{46.8} & \scib{7.5}{-10} \\
Obstacles (6) & \textbf{46.6} & \sci{5.9}{-3} & 3909.3 & \textbf{0}~$\dagger$ & --- & --- & --- & --- & \textbf{46.2} & \sci{2.0}{-4} & \textbf{48.6} & \sci{1.6}{-5} \\
Walker2D & \textbf{134.2} & \sci{8.1}{-3} & 8336.6 & \textbf{0} & --- & --- & --- & --- & \textbf{138.3} & \sci{9.8}{-4} & \textbf{137.0} & \scib{1.2}{-8} \\
Hopper & \textbf{134.0} & \sci{6.2}{-3} & 4494.4 & \textbf{0} & --- & --- & --- & --- & \textbf{135.3} & \sci{6.7}{-4} & \textbf{135.6} & \scib{6.0}{-9} \\
\bottomrule
\end{tabular}
}
\caption{Mean per-sample generation time (ms) and constraint violation across 20 samples.
Violation is $\norm{g(x_1)}$ for equality-constrained problems (top) and
$\norm{\relu(h(x_1))}$ for inequality-constrained problems (bottom).
Bold times are within 20\% of the lowest time in that row; bold violations are below $10^{-6}$
(near machine precision for single-precision float).
$^\dagger$At least one NaN error produced during generation; average shown is across non-NaN samples.}
\vspace{-1.05em}
\label{tab:generation_times}
\end{table*}

\section{Conclusion}\label{sec:conclusion}

Lagrangian dual flows (LDF) provide a simple, computationally efficient means to apply inference-time constraints to pre-trained flow matching policies. By flowing a dual state along with the samples, we can guarantee constraint satisfaction under standard constraint-qualification assumptions and a small-variation assumption without resorting to projections, pseudoinverses, or optimization subproblems. In fact, with careful implementation, a single vector-Jacobian product is all that is needed to implement LDF. 
Constraints need not be linear or even invertible. 
This makes LDF a promising candidate for emerging flow matching applications like robotics, planning, and control where constraints are critical but complex, and inference-time compute is limited.

Future work will also focus on improving efficiency, reducing ODE stiffness to enable few-step generation, and characterizing the conditional distribution of constrained samples.
More broadly, LDF draws a direct connection between flow matching and the classical theory of primal-dual and saddle point dynamics. We believe that the literature on continuous-time optimization offers tools that may translate into faster and tighter constrained generation. Pursuing this connection and moving from feasibility to constraint-conditioned generation remains an exciting avenue for future work.

\bibliographystyle{iclr2027_conference}
\bibliography{references}

\appendix

\section{Implementation Details}\label{app:implementation}

Our implementation of LDF, along with baselines and code for reproducing all figures in this paper is available at \url{https://anonymous.4open.science/r/constrained\_flow\_matching}. We use JAX and Flax NNX for model training and deployment; ODEs are solved with Diffrax \citep{kidger2021on}. 

\subsection{Constrained Generation Methods}

\subsubsection{Lagrangian Dual Flows}

We implement Algorithm \ref{alg:equality_constrained} in JAX, leaving the penalty $c$ and dual flow exponent $p$ as free parameters. The correction terms in \eqref{eq:x} are implemented using a single VJP: the explicit Jacobian $J_g(x)$ is never formed or stored in memory. 

We wrap the dynamics of the sample $x$, Lagrange multipliers $\lambda$, and any slack variables $s$ into a single ODE, and solve this ODE with Diffrax. For fixed-step integration, we use the second-order midpoint method. For error-controlled integration, we use Diffrax's default step-size controller with a minimum step size of $10^{-5}$.

Since the ODE is well-defined only as $t \to 1^{-}$, we integrate only to $1 - dt$, where $dt$ is the step size. Additionally, we add a small constant $\epsilon = \sci{1}{-6}$ to the dual rescaling term such that \eqref{eq:lambda} is implemented as $\dot{\lambda}_t = g(x_t)/(1-t+\epsilon)^p$.

\subsubsection{Penalty Only Baseline}

The penalty-only baseline exercises the same code path as our proposed approach, but we fix $\lambda = 0$. The net ODE is therefore 
\begin{equation}
    \dot{x}_t = v_\theta(x_t, t) - \frac{c}{2}\, \nabla_x \|g(x_t)\|^2,
\end{equation}
a simple quadratic penalty guiding the generation process. For inequality constraints, we instead use the ODE
\begin{equation}
    \dot{x}_t = v_\theta(x_t, t) - \frac{c}{2}\, \nabla_x \|\relu(h(x_t))\|^2,
\end{equation}
which penalizes violations in the inequality constraint $h(x_1) \leq 0$.

\subsubsection{Pseudoinverse Guidance Baseline}

For pseudoinverse guidance, we use JAX and Diffrax to flow samples according to the ODE defined in \cite[Algorithm 1]{pokle2024training}.
Note that this algorithm is designed for linearly constrained problems of the form 
\begin{equation}
    A x = y.
\end{equation}
To support generic nonlinear constraints $g(x) = 0$, we simply replace $A$ with the Jacobian $\jac{g}(x_t)$ and the residual $A x - y$ with $g(x)$.

To maximize runtime performance, we use LU decomposition and a linear solve to avoid full computation of the Moore-Penrose pseudoinverse, and a single VJP to avoid storing the full Jacobian of last-state predictions with respect to the current estimate.

\subsubsection{Physics-Constrained Flow Matching Baseline}

We similarly implement \cite[Algorithm 1]{utkarsh2025pcfm} using JAX and Diffrax. This PCFM algorithm is considerably more complex than other baselines, involving several interleaved steps of ODE solving, projection, and optimization subproblems. We solve the optimization subproblem \cite[Algorithm 1, line 8]{utkarsh2025pcfm} with 10 steps of gradient descent at learning rate 0.1. We found that PCFM's generation paths were quite sensitive to the projection step \cite[Algorithm 1, line 6]{utkarsh2025pcfm}. In particular, performing this projection to a linearization of the constraint manifold resulted in extremely noisy generation paths for the unit circle constraint, as shown in Fig.~\ref{fig:pcfm_projection_iters} below. 

\begin{figure}
    \centering
    \includegraphics[width=0.7\linewidth]{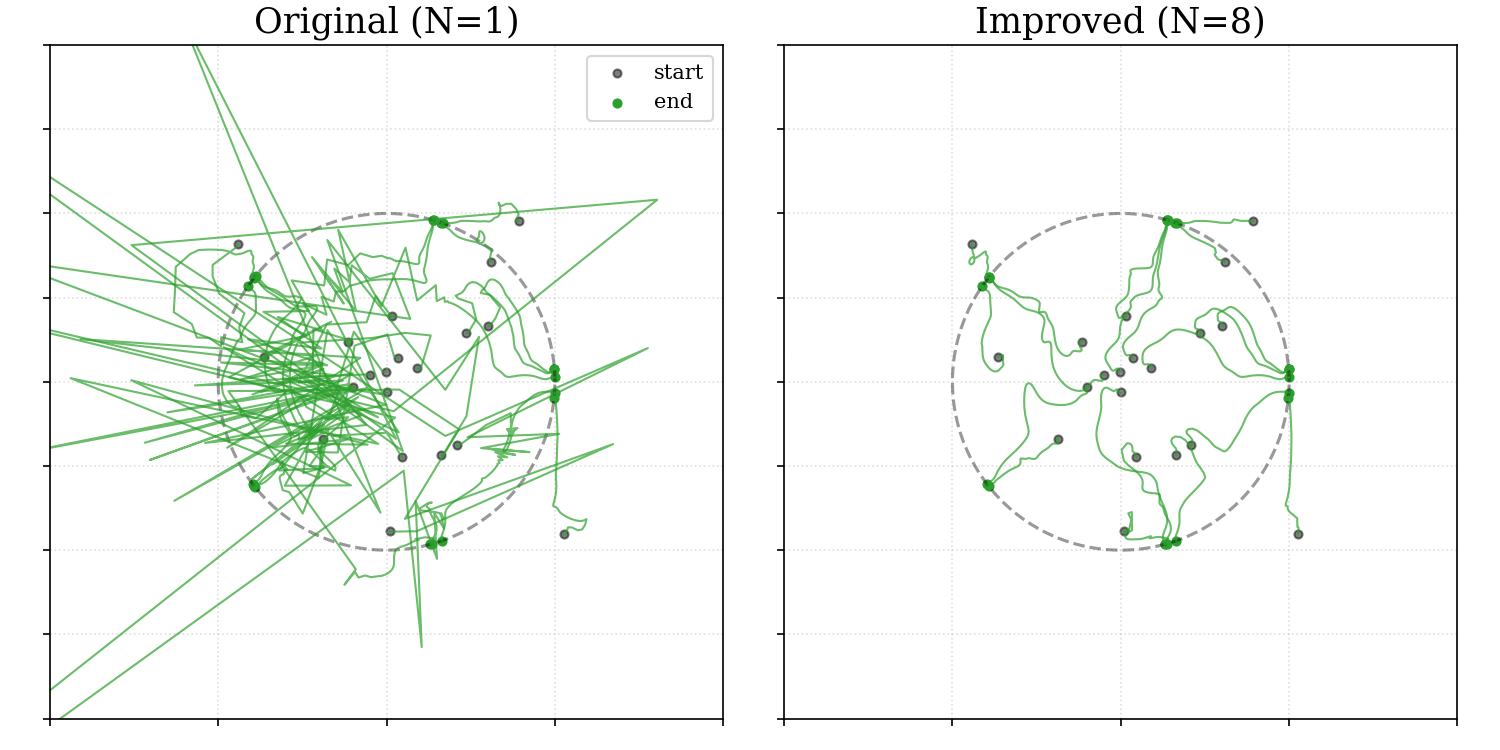}
    \caption{The original PCFM algorithm (left, \cite[Algorithm 1]{utkarsh2025pcfm}) produces extremely noisy generation paths for our star examples, subject to a unit circle constraint. We compare with an improved baseline (right), where increasing the number of projection steps avoids overshoot and results in more reasonable generation paths.}
    \label{fig:pcfm_projection_iters}
\end{figure}

To avoid this issue, our implementation repeats this projection step $N$ times, re-linearizing the constraint after each projection. We found that $N=8$ offered a reasonable balance of performance and computation speed for our examples. We use the improved $N=8$ version as our baseline for all of our numerical experiments.

\subsubsection{Control Barrier Function (CBF) Baseline}

For the CBF baseline, we implement \citep[Algorithm 1]{safeflow2025} in JAX, using QPAX \citep{arrizabalaga2026differentiableinteriorpointmethodsingle} to solve the CBF-QP. This quadratic program is not feasible in general, so we use the relaxed (elastic-mode) formulation to solve it. Note that the analytical QP solution described in \citep{safeflow2025} is not applicable in our case of general inequality constraints.

Note that the need for this relaxation is an indication that we do not have a true control barrier function for the problem at hand. Finding such a barrier function is difficult, requiring careful analysis and a priori knowledge of the constraint set and system dynamics, much like constructing a Lyapunov function. The relaxed solution that we use makes a ``best effort'' attempt to satisfy these conditions, but they are infeasible at some point in the generation process for all of our examples. As a result, the CBF approach cannot guarantee constraint satisfaction, and in practice sometimes violates constraints or results in divergence during the denoising process.

\subsection{Example Details}

Here we provide a brief summary regarding our example systems. Full details can be found in the source code at \url{https://anonymous.4open.science/r/constrained_flow_matching}. Note that our pretrained models and architectures are not particularly optimized or finely tuned; their primary purpose is rather to provide a testbed for applying inference-time constraints.

\subsubsection{Equality Constrained Star}

We first consider a simple example with nonlinear equality constraints. In particular, we pre-train a small feed-forward network to generate samples from a star-shaped distribution, shown in the leftmost column of Fig.~\ref{fig:constrained_star}. This flow model $v_\theta(x, t)$ is a simple MLP with 4 hidden layers of 64 units and swish activations. We train on 1024 data points generated on a 2D star manifold. Training this small network takes only a few minutes on a laptop.

We then impose a unit circle constraint
\begin{equation*}
    g(x) = \|x\|^2 - 1
\end{equation*}
at inference time. All methods use 100 steps ($dt=0.01$) of the second-order midpoint rule for ODE integration.

\subsubsection{Inequality Constrained Star}

To illustrate LDF's ability to handle inequality constraints, we impose a simple constraint on the same  pretrained star model, 
\begin{equation*}
    h\left(\begin{bmatrix}x^0\\x^1\end{bmatrix}\right) = -x^0,
\end{equation*}
which constrains generated samples to the right-hand side of the plane. The resulting samples are shown in Fig.~\ref{fig:inequality_constrained_examples}a. We again use 100 denoising steps and midpoint integration.

Adding this constraint is computationally cheap: generating 1000 LDF-constrained samples (in parallel, 100 denoising steps, midpoint integration) takes 0.0121 seconds, compared to 0.0114 seconds for unconstrained generation.

\subsubsection{MNIST Inpainting}\label{subsec:MNIST}

Image inpainting, where only a certain subset of pixels are generated by a pre-trained model, is a special case of equality constraints $g(x) = 0$. In this case, $g(x)$ is a linear constraint specifying the value of fixed pixels. We train a $\sim$6M parameter UNet to generate MNIST digits, then apply an inpainting constraint, which fixes the top half of the image. 

\begin{figure}
    \centering
    \includegraphics[width=0.9\linewidth]{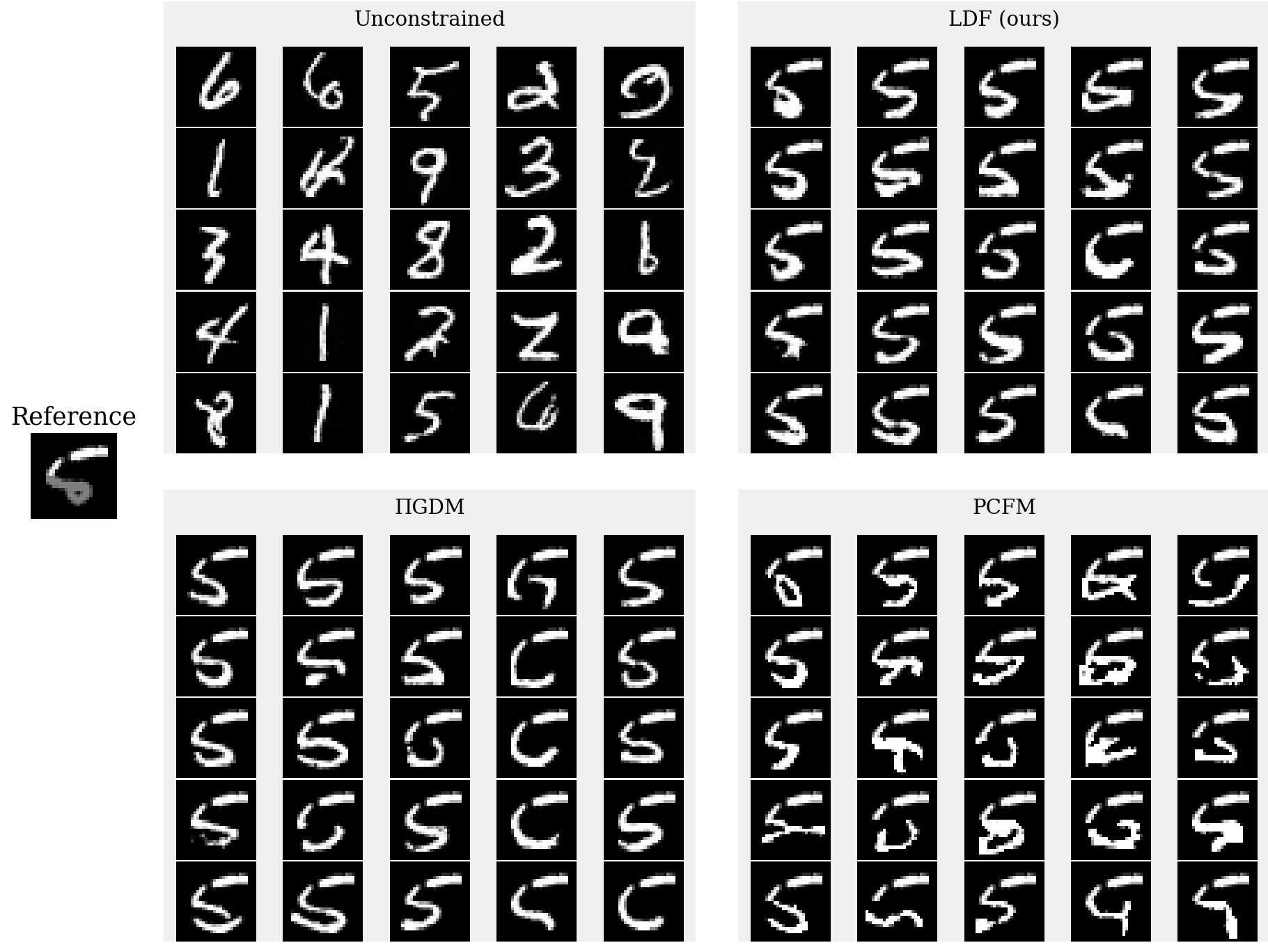}
    \caption{Constrained samples generated with different random seeds.}
    \label{fig:mnist_extended}
\end{figure}

Figure~\ref{fig:mnist_extended} shows samples generated with this inpainting constraint.
Both pseudoinverse guidance \citep{pokle2024training} and our approach produce visually plausible images. Interestingly, while physics-constrained flow matching \citep{utkarsh2025pcfm} produces the tightest constraint satisfaction (Table~\ref{tab:generation_times}), clear artifacts are visible in the resulting images. This suggests that the projection substeps used in PCFM may be overly aggressive, pushing samples off of the data manifold in an attempt to enforce constraint satisfaction.

For the pretrained model, we use an encoder/decoder U-net with a (64, 128, 256) dimensional channel schedule. We use a 128-dim sinusoidal time embedding, conditioned via AdaGN \citep{dhariwal2021diffusion}. This model is trained to generate unconditional samples: classification labels are not used during training. Training this model takes about 30 minutes. At inference time, we use 100 denoising steps.

\subsection{Obstacle Avoidance}

For obstacle avoidance, the unconstrained temporal UNet model is trained to produce a variety of paths between a start and a goal configuration. Samples from the unconstrained model are shown in the left column of Fig.~\ref{fig:obstacle_comparison}. At test time, a random set of circular obstacles is generated. Configurations are rejected if obstacles are overlapping or too close; see the code for full details. 

Obstacle avoidance constraints are formulated by sampling evenly spaced points along each path, which is represented as a Bezier spline. We use 500 steps of the midpoint rule for this example.

Constrained samples are shown in Fig.~\ref{fig:obstacle_comparison}. Beyond requiring an expensive QP subproblem at each iteration of the denoising process, the CBF method frequently produces collisions and divergences (e.g., NaN errors during the denoising process). In contrast, LDF requires simply integrating a standard ODE and produces reliable performance as the number of obstacles increases.

\begin{figure}
    \centering
    \includegraphics[width=0.8\linewidth]{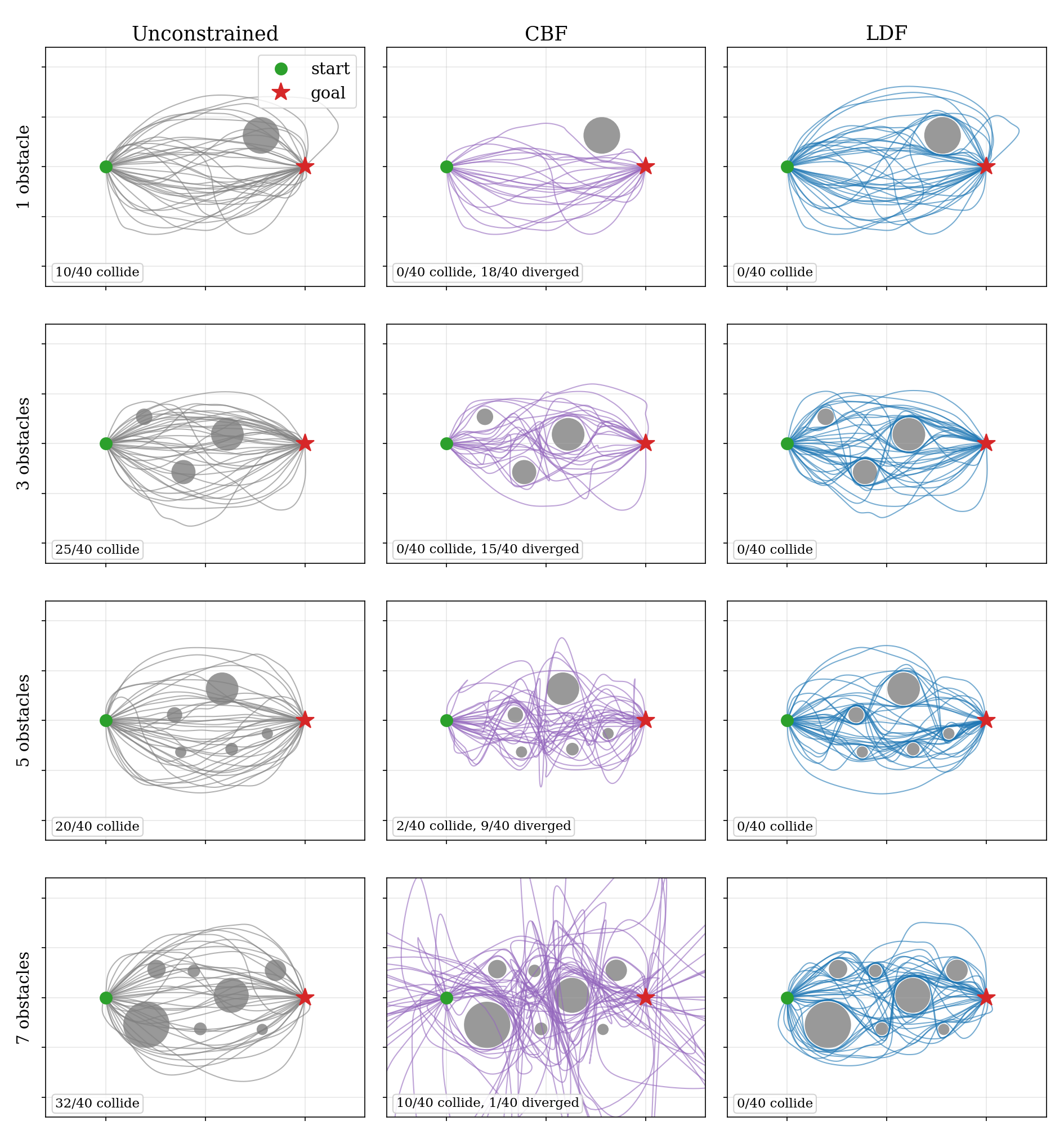}
    \caption{Collision-free path planning comparison between LDF and CBF methods. }
    \label{fig:obstacle_comparison}
\end{figure}

\subsection{Legged Locomotion}

\begin{figure}
    \centering
    \includegraphics[width=0.8\linewidth]{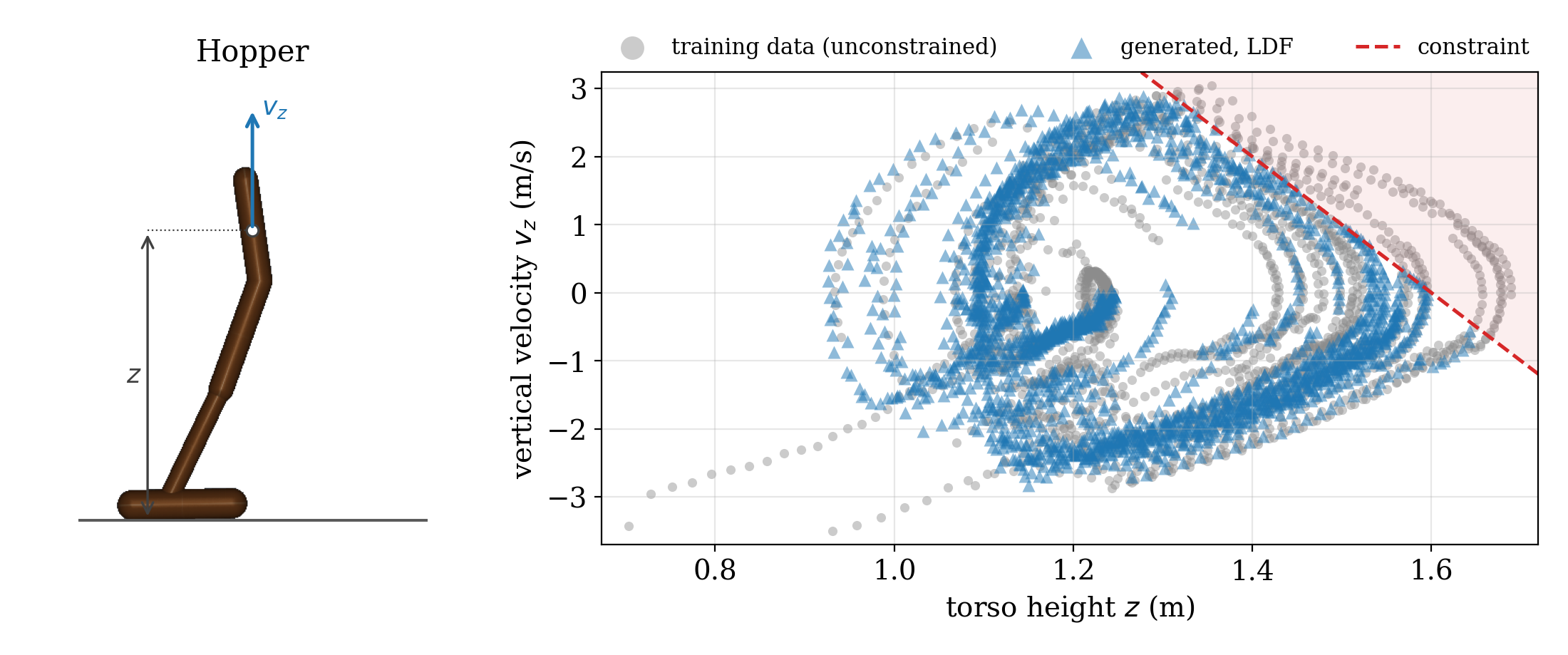}
    
    \includegraphics[width=0.8\linewidth]{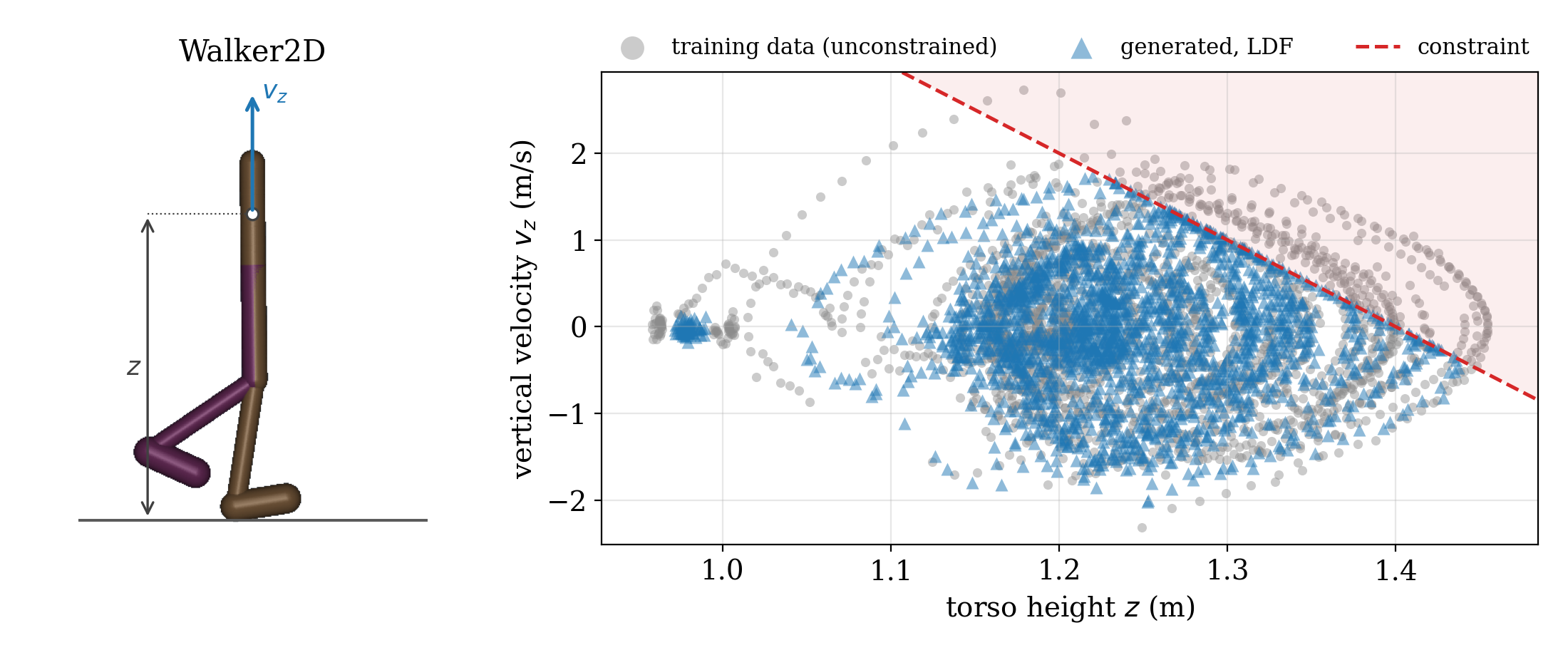}
    
    \caption{Illustrations and constrained phase portraits for the hopper (top) and walker (bottom) examples. LDF preserves the general character of the pretrained trajectories while enforcing the constraint.}
    \label{fig:locomoation}
\end{figure}

For the hopper and walker examples, we train flow models to generate trajectories based on the D4RL \citep{fu2020d4rl} dataset. Our setup follows \citep{xiao2023safediffuser, yang2026safeflowmatcher}. In particular, each trajectory consists of actions (3 actions for the hopper, 6 for the walker) and observations (11 for the hopper, 17 for the walker) over a 32-step horizon. The unconstrained models are ~3.3M parameter temporal UNets.

We adopt the constraint, a velocity-dependent maximum height, from \citep{yang2026safeflowmatcher}. In particular, the robot height $z$ must remain below a threshold $h_r - \varphi v_{z}$, where $h_r$ and $\varphi$ are constants. The hopper uses $(h_r = 1.6, \varphi = 0.1)$ while the walker uses $(h_r = 1.4, \varphi = 0.1)$.

\section{Additional Experiments}\label{app:extra_experiments}

\subsection{Tightening Constraints}

Our theoretical analysis indicates that exact constraint satisfaction is achieved (under sufficient regularity conditions) for $p=2$ and an exact solution of the ODE. We find that performance in practice depends on choices of coefficients $c$ and $p$, as well as integration accuracy.

\begin{figure}
    \centering
    \includegraphics[width=0.6\linewidth]{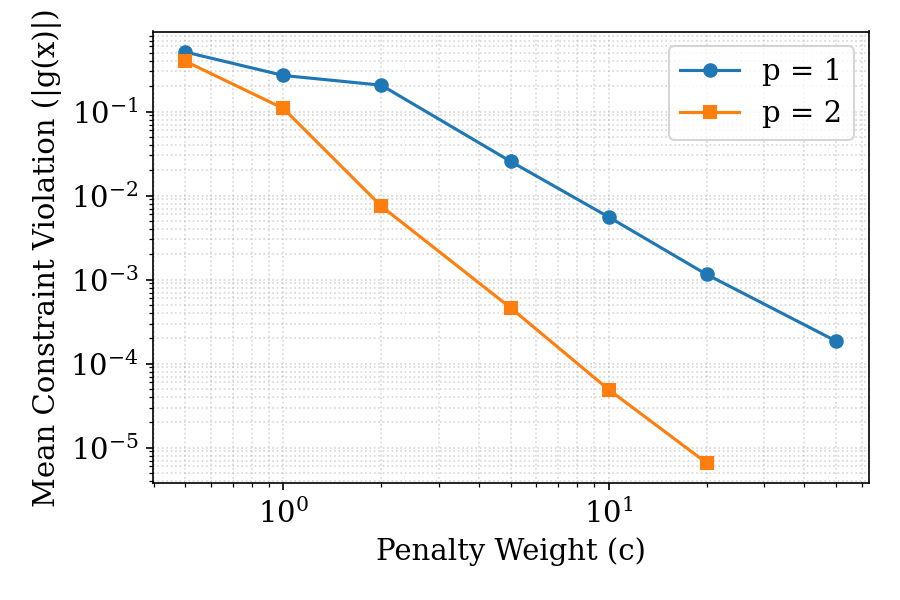}
    \caption{Average constraint violation for the star example, subject to a unit circle constraint. }
    \label{fig:violation_vs_penalty}
\end{figure}

We investigate impact of the penalty weight $c$ and time-rescaling factor $p$ on constraint violation. To minimize the impact of integration errors, we use a 5th order Runge-Kutta scheme with error-controlled integration at accuracy $10^{-5}$, as implemented in Diffrax \citep{kidger2021on}. Figure~\ref{fig:violation_vs_penalty} shows that constraint violation in the final generated samples decreases predictably with the penalty weight $c$, and that using $p = 2$ provides tighter constraint satisfaction for the same penalty weight.

While $p=1$ produces reasonable performance in this case, there are cases in which $p=1$ fails to enforce the constraint for any $c$. See Appendix~\ref{app:p12} below for details.

\subsection{ODE Stiffness}

Enforcing constraints with LDF does produce a stiffer ODE than the original unconstrained model. While increasing the penalty $c$ can tighten constraint satisfaction, as shown in Fig.~\ref{fig:violation_vs_penalty} above, this also makes the ODE stiffer and more difficult to integrate.

For fixed-step integration schemes, this results in a tradeoff: higher penalties will produce a stiff-enough ODE that numerical integration error dominates, ultimately increasing constraint violation. This is shown for the MNIST inpainting example in Fig.~\ref{fig:mnist_violation}. Error-controlled integration methods \citep{hairer1993solving} avoid this problem by adapting the step size to bound the impact of numerical error. Of course, these adaptive integration methods also tend to introduce some computational overhead.

Developing LDF formulations that limit the stiffness of the denoising ODE is an important area for future research. 

\begin{figure}
    \centering
    \includegraphics[width=0.6\linewidth]{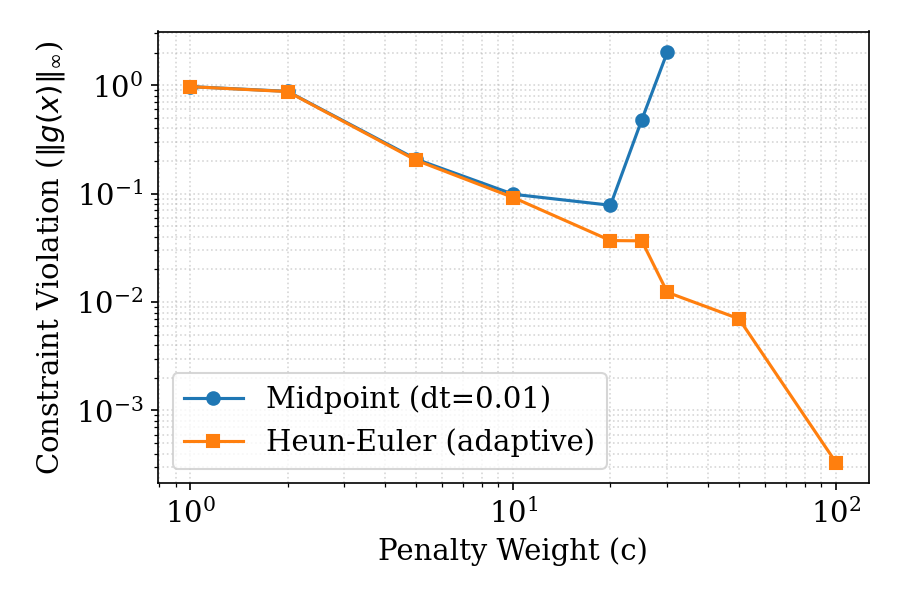}
    \caption{For MNIST image inpainting, constraint tightness saturates around $c=20$ with a fixed-step midpoint integrator. Error-controlled integration (in this case adaptive-step Heun-Euler) \citep{hairer1993solving} allows us to further tighten constraint satisfaction.}
    \label{fig:mnist_violation}
\end{figure}

\subsection{Role of Dual Dynamics}

To illustrate the role of the flowed dual $\lambda_t$, we compare with the penalty-only approach where we ignore the Lagrangian term. This is similar to the quadratic penalty approach used in \citep{ho2022video, chung2022mcg}. Using error-controlled integration as above, we generate samples using a range of penalty weights $c$. We record the resulting constraint violation, averaged over 200 samples, and number of denoising steps required. The number of denoising steps is a rough proxy for ODE stiffness: a more difficult ODE will require more steps to resolve to the requested accuracy. 

\begin{figure}
    \centering
    \includegraphics[width=0.6\linewidth]{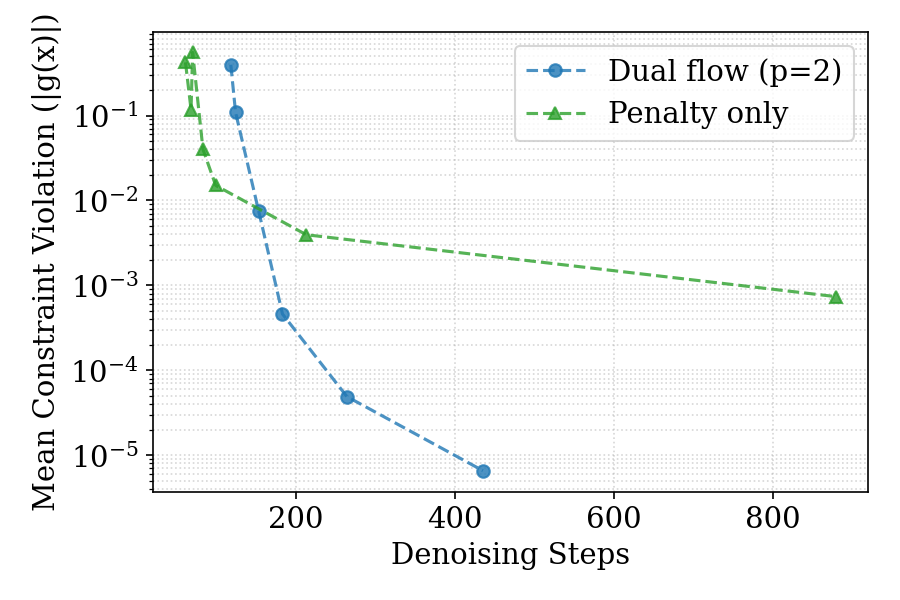}
    \caption{Comparison with a simple penalty method (e.g., $\lambda = 0$). Flowed Lagrangian dynamics allows us to further reduce constraint violations on MNIST inpainting while limiting the stiffness of the ODE.}
    \label{fig:steps_vs_violation}
\end{figure}

The resulting data is shown in Fig.~\ref{fig:steps_vs_violation}. Even with the stiffest penalty $c=50$, the penalty-only method is only able to reduce constraint violation to around $10^{-3}$. More critically, the penalty-only dynamics in this case become extremely stiff, requiring over 800 denoising steps to resolve. Meanwhile, LDF allows us to decrease constraint violations by two orders of magnitude with half the denoising steps. 

\section{Proofs}\label{app:proofs}
\newcommand{\abs}[1]{\left|#1\right|}
\newcommand{\diff}{\mathrm{d}}
\newcommand{\op}{2}

\subsection{Proof of Theorem~\ref{thm:p2} }

\begin{proof}
The proof has four steps. First, we remove the singularity at $t=1$ by an
exponential time change. Second, we construct a Lyapunov function and verify
its positive definiteness and its exact dissipation along the autonomous
limit. Third, we decompose the Lyapunov function's derivative along the true dynamics into the
autonomous part plus two controlled corrections. Finally, we apply the
comparison lemma to the resulting inequality, obtaining an explicit decay rate.

Throughout we abbreviate
\begin{equation*}
A(x) := \jac{g}(x)\trasp{\jac{g}(x)} \in \real^{m\times m},
\qquad
b(x,t) := \jac{g}(x)\,v_\theta(x,t) \in \real^m.
\label{eq:Ab-def}
\end{equation*}
 
\medskip
\noindent\textbf{Step 1 (Exponential time change).}
Introduce $\tau := -\log(1-t)$, mapping $[0,1)$ onto $[0,\infty)$ with
$1-t=e^{-\tau}$ and $\diff\tau/\diff t = e^{\tau}$, and the rescaled dual
variable
\begin{equation}
\eta(\tau) := (1-t)\lambda(t) = e^{-\tau}\lambda.
\label{eq:eta-def}
\end{equation}
The chain rule applied to $g(x_t)$ along~\eqref{eq:x} gives
$\dot g = \jac{g}(x)\dot x = b(x,t) - A(x)\lambda - c\,A(x)g$. Converting to
$\tau$-time via $\diff g/\diff\tau = e^{-\tau}\dot g$ and $\lambda = e^{\tau}\eta$,
\begin{equation}
\frac{\diff g}{\diff\tau}
= -A(x)\,\eta + e^{-\tau}\big[b(x,t) - c\,A(x)\,g\big].
\label{eq:gdot-tau}
\end{equation}
For the dual variable, $\dot\lambda = g/(1-t)^2$ gives
$\diff\lambda/\diff\tau = e^{-\tau}\dot\lambda = e^{\tau}g$, so differentiating
\eqref{eq:eta-def},
\begin{equation}
\frac{\diff\eta}{\diff\tau}
= -e^{-\tau}\lambda + e^{-\tau}\frac{\diff\lambda}{\diff\tau}
= -\eta + g.
\label{eq:etadot-tau}
\end{equation}
Under assumption~\Aref{4}, as $\tau\to\infty$ the factor $e^{-\tau}$ vanishes, and \eqref{eq:gdot-tau}-\eqref{eq:etadot-tau}
are close to the autonomous linear system $\diff g/\diff\tau = -A_*\eta$,
$\diff\eta/\diff\tau = g-\eta$, with $A(x)$ frozen at $A_*$. The deviations
$A(x)-A_*$ and the $e^{-\tau}$ forcing are treated as corrections in Step~3.
 
\medskip
\noindent\textbf{Step 2 (Lyapunov function).}
With the reference matrix $A_*\succ0$ from Assumption~\Aref{4}, define the candidate Lyapunov function
\begin{equation}
V(g,\eta) := \trasp{g}\big(A_*^{-1}+\tfrac{1}{2} I_m\big)g
- \trasp{g}\eta + \tfrac{1}{2}\norm{\eta}^2 + \tfrac{1}{2}\trasp{\eta}A_*\eta.
\label{eq:V-def}
\end{equation}
 
First, we establish that this candidate is positive definite.
Stacking $w:=(g,\eta)\in\R^{2m}$, we have $V=\frac{1}{2}\trasp{w}Pw$ with
\begin{equation}
P = \begin{pmatrix} 2A_*^{-1}+I_m & -I_m \\ -I_m & I_m+A_* \end{pmatrix}.
\label{eq:P-def}
\end{equation}
Diagonalizing
$A_*=Q\,\mathrm{diag}(\mu_1,\dots,\mu_m)\,\trasp Q$ with $Q$ orthogonal and
conjugating $P$ by $\mathrm{diag}(Q,Q)$ block-diagonalizes $P$ into $m$
blocks
\begin{equation}
P_i = \begin{pmatrix} 2\mu_i^{-1}+1 & -1 \\ -1 & 1+\mu_i \end{pmatrix}.
\label{eq:Pi-def}
\end{equation}
Each $P_i\succ 0$ since it is strictly diagonally dominant. Hence $P\succ 0$, so $V$ is positive definite and
\begin{equation}
c_1\big(\norm{g}^2+\norm{\eta}^2\big) \leq V(g,\eta) \leq c_2\big(\norm{g}^2+\norm{\eta}^2\big),
\qquad c_1:=\tfrac{1}{2}\lambda_{\min}(P),\ c_2:=\tfrac{1}{2}\lambda_{\max}(P).
\label{eq:V-bounds}
\end{equation}
 
\emph{Dissipation along the autonomous limit.}
Along $\diff g/\diff\tau=-A_*\eta$, $\diff\eta/\diff\tau=g-\eta$ we claim
\begin{equation}
\frac{\diff V}{\diff\tau}\Big|_{\mathrm{autonomous}} = -\norm{g}^2-\norm{\eta}^2.
\label{eq:auto-dissip}
\end{equation}
Differentiating \eqref{eq:V-def} term by term:
\begin{align*}
\frac{\diff}{\diff\tau}\!\Big[\trasp{g}(A_*^{-1}+\tfrac{1}{2} I)g\Big]
&= 2\trasp{g}(A_*^{-1}+\tfrac{1}{2} I)(-A_*\eta) = -2\trasp{g}\eta-\trasp{g}A_*\eta,\\
\frac{\diff}{\diff\tau}\!\Big[-\trasp{g}\eta\Big]
&= -(-A_*\eta)^{\!\top}\eta - \trasp{g}(g-\eta) = \trasp{\eta}A_*\eta-\norm{g}^2+\trasp{g}\eta,\\
\frac{\diff}{\diff\tau}\!\Big[\tfrac{1}{2}\norm{\eta}^2\Big]
&= \trasp{\eta}(g-\eta) = \trasp{g}\eta-\norm{\eta}^2,\\
\frac{\diff}{\diff\tau}\!\Big[\tfrac{1}{2}\trasp{\eta}A_*\eta\Big]
&= \trasp{\eta}A_*(g-\eta) = \trasp{g}A_*\eta-\trasp{\eta}A_*\eta,
\end{align*}
using the symmetry of $A_*$ for $(A_*\eta)^{\!\top}\eta=\trasp\eta A_*\eta$
and $\trasp\eta A_* g=\trasp g A_*\eta$. Summing, we are left with $-\norm{g}^2-\norm{\eta}^2$,
which is \eqref{eq:auto-dissip}.
 
\medskip
\noindent\textbf{Step 3 (Decomposition along the true dynamics).}
Now we differentiate $V$ along \eqref{eq:gdot-tau}-\eqref{eq:etadot-tau},
where $A(x)$ replaces $A_*$ and the forcing $e^{-\tau}[b-cAg]$ is present.
With $\nabla_g V = 2(A_*^{-1}+\frac{1}{2} I)g-\eta$ and $\nabla_\eta V = -g+\eta+A_*\eta$,
substituting and decomposing $A(x)=A_*+(A(x)-A_*)$ gives
\begin{equation}
\frac{\diff V}{\diff\tau}
= -\norm{g}^2-\norm{\eta}^2 + R_A + e^{-\tau}R_p,
\label{eq:dV-decomp}
\end{equation}
where the autonomous part is exactly \eqref{eq:auto-dissip}, and
\begin{align}
R_A &:= -2\trasp{g}\big(A_*^{-1}+\tfrac{1}{2} I\big)(A(x)-A_*)\eta
+ \trasp{\eta}(A(x)-A_*)\eta,
\label{eq:RA-def}\\
R_p &:= \big(2\trasp{g}\big(A_*^{-1}+\tfrac{1}{2} I\big)-\trasp{\eta}\big)\xi,
\qquad \xi := b(x,t)-c\,A(x)\,g.
\label{eq:Rp-def}
\end{align}
 
\emph{Bound on $R_A$.}
By Cauchy--Schwarz on each term of \eqref{eq:RA-def} and Young's inequality
$2\norm{g}\norm{\eta}\leq\norm{g}^2+\norm{\eta}^2$,
\begin{align*}
\abs{R_A}
&\leq 2\norm{g}\,\norm{A_*^{-1}+\tfrac{1}{2} I}_2\norm{A(x)-A_*}_2\norm{\eta}
+ \norm{A(x)-A_*}_2\norm{\eta}^2\\
&\leq \norm{A(x)-A_*}_2\Big[\norm{A_*^{-1}+\tfrac{1}{2} I}_2\big(\norm{g}^2+\norm{\eta}^2\big)+\norm{\eta}^2\Big]\\
&\leq \big(\norm{A_*^{-1}+\tfrac{1}{2} I}_2+1\big)\norm{A(x)-A_*}_2\big(\norm{g}^2+\norm{\eta}^2\big)
\\ &= K_{A_*}\norm{A(x)-A_*}_2\big(\norm{g}^2+\norm{\eta}^2\big),
\end{align*}
using $\norm{\eta}^2\leq \norm{g}^2+\norm{\eta}^2$. By Assumption~\Aref{4},
$K_{A_*}\norm{A(x)-A_*}_2\leq\rho$, so
\begin{equation}
\abs{R_A} \leq \rho\big(\norm{g}^2+\norm{\eta}^2\big).
\label{eq:RA-final}
\end{equation}
 
\emph{Bound on $R_p$.}
The vector $\xi$ is bounded: by \Aref{1}-\Aref{3},
$\norm{\xi}\leq\norm{b}+c\norm{A}_2\norm{g}\leq LB+c\,a_{\max}G=:M_\xi$.
Then by Cauchy-Schwarz,
\begin{equation}
\abs{R_p}
\leq 2\norm{g}\,\norm{A_*^{-1}+\tfrac{1}{2} I}_2\norm{\xi}+\norm{\eta}\norm{\xi}
\leq K_p\big(\norm{g}+\norm{\eta}\big),
\label{eq:Rp-final}
\end{equation}
with $K_p:=M_\xi\max\{2\norm{A_*^{-1}+\tfrac{1}{2} I}_2,1\}$ depending only on
$L,B,c,a_{\max},A_*,G$.
 
\emph{Combined inequality.}
Substituting \eqref{eq:RA-final} and \eqref{eq:Rp-final} into
\eqref{eq:dV-decomp},
\begin{equation}
\frac{\diff V}{\diff\tau}
\leq -(1-\rho)\big(\norm{g}^2+\norm{\eta}^2\big)
+ e^{-\tau}K_p\big(\norm{g}+\norm{\eta}\big).
\label{eq:dissip}
\end{equation}

\noindent\textbf{Step 4 (Comparison lemma and explicit rate).}
Write $\norm w^2 := \norm g^2 + \norm\eta^2$. Since $\norm g + \norm\eta \leq \sqrt{2}\,\norm w$, the combined inequality~\eqref{eq:dissip} gives
\begin{equation*}
\frac{\diff V}{\diff\tau} \leq -(1-\rho)\,\norm w^2 + e^{-\tau} \sqrt{2}\,K_p \norm w.
\end{equation*}

Set $y := \sqrt V \ge 0$. Where $V > 0$, using $\norm w^2 \ge V/c_2 = y^2/c_2$ on the dissipation term and $\norm w \leq \sqrt{V/c_1} = y/\sqrt{c_1}$ on the forcing term (both from~\eqref{eq:V-bounds}),
\begin{equation*}
\frac{\diff y}{\diff\tau} = \frac{1}{2\sqrt V}\frac{\diff V}{\diff\tau}
\leq -\alpha y + \gamma e^{-\tau},
\qquad \alpha := \frac{1-\rho}{2c_2} > 0, \quad \gamma := \frac{K_p}{\sqrt{2 c_1}} > 0 .
\end{equation*}
The comparison lemma~\cite[Lemma~3.4]{khalil2002nonlinear} then gives, for all $\tau \geq 0$,
\begin{equation}
y(\tau) \leq y(0) e^{-\alpha\tau} + \gamma\int_0^\tau e^{-\alpha(\tau-u)} e^{-u}\,\diff u
= y(0) e^{-\alpha\tau} + \gamma
\begin{cases}
\dfrac{e^{-\alpha\tau} - e^{-\tau}}{1-\alpha}, & \text{if }\alpha \neq 1,\\[1.5ex]
\tau\,e^{-\tau}, & \text{if }\alpha = 1.
\end{cases}
\label{eq:comparison-eq}
\end{equation}
Both terms on the right of~\eqref{eq:comparison-eq} decay so we can conclude that $y(\tau) \to 0$ as $\tau \to \infty$. 

Additionally, for all $i \in \until{m}$, $\lambda_{\max}(P)\geq P_{i}^{11}=2\mu_i^{-1}+1>1$ for a symmetric matrix, and
$2c_2=\lambda_{\max}(P)$ with $\rho\ge0$, we have
$\alpha=(1-\rho)/\lambda_{\max}(P)<1$ and thus there is a constant $M > 0$ such that $y(\tau) \leq M\, e^{-\alpha \tau}$ for all $\tau \geq 0$. 
Converting to original time via $e^{-\tau} = 1-t$ and using $\norm{g(x_t)} \leq \norm w \leq y/\sqrt{c_1}$,
\begin{equation*}
\norm{g(x_t)} \leq \frac{M}{\sqrt{c_1}}\,(1-t)^\alpha = O\big((1-t)^\alpha\big)
\quad\text{and}\quad
(1-t)\norm{\lambda_t} = \norm\eta = O\big((1-t)^\alpha\big)
\end{equation*}
as $t \to 1^-$. In particular both tend to zero.
\end{proof}

\subsection{Proof of Lemma~\ref{lem:posinv} and Theorem~\ref{thm:ineq}}

\begin{proof}[Proof of Lemma~\ref{lem:posinv}] If the $i$-th component of $s_t$, $s_t^i$, is equal to zero at some time $t \in [0,1)$, then from~\eqref{eq:ineq-s}, $\dot{s}_t^i = c\; \relu(-h^i(x_t) - \lambda_t^i/c) \geq 0$. Hence $s_t^i$ cannot decrease past zero and with $s_0 \geq 0$, we conclude that $s_t \geq 0$ for all $t \in [0,1)$.    
\end{proof}

For self-containedness, we reproduce the assumptions for convergence here, stated in terms of $h$. 
We assume the following hold along the trajectory of
\eqref{eq:ineq-x}-\eqref{eq:ineq-lambda}.
\begin{itemize}
\item[(B1)]\phantomsection\label{assm:B1} \emph{Smoothness:} $h\in C^2(\R^n;\R^k)$ and $v_\theta\in C^1(\R^n \times [0,1); \R^n)$,
with $\norm{\jac{h}(x)}_2\leq L$ and $\norm{v_\theta(x,t)}\leq B$;
\item[(B2)]\phantomsection\label{assm:B2} \emph{Constraint qualification:} $a_{\min}I_k\preceq \jac{h}(x)\jac{h}(x)^\top \preceq a_{\max}I_k$
for constants $0<a_{\min}\leq a_{\max}<\infty$.
\item[(B3)]\phantomsection\label{assm:B3} \emph{Bounded trajectory:} the trajectory $(x_t,s_t,\lambda_t)$
exists on $[0,1)$ with $x_t$ and $s_t$ in a compact set. 
\end{itemize}

These assumptions are identical to those assumed in the equality-constrained case with the additional assumption that the slack variable $s_t$ evolves in a compact set. This additional assumption can be enforced by replacing the $\relu(-h(x_t)-\lambda_t/c)$ in~\eqref{eq:ineq-s} with $\min\{\relu(-h(x_t)-\lambda_t/c), R\}$ for some $R > 0$. This enforces boundedness at the expense of enforcing the two-sided constraint $-R \leq h(x_1) \leq 0$ for the generated sample where we understand $-R \leq h(x_1)$ to mean that each entry of $h(x_1)$ is greater than or equal to $-R$.

Unlike the equality case, the soft projection switches between active sets, so the proof relies on a Lyapunov function that dissipates simultaneously in every possible regime. Because the dissipation margin $\kappa_0$ it produces enters the small-variation condition, we state it first.

First, for notation, for a subset $\mathcal{S} \subseteq \until{k}$, let $E_{\mathcal{S}}:=\mathrm{diag}(\mathbf 1_{i\in\mathcal{S}})$
be the corresponding $0/1$ diagonal matrix.

\begin{lemma}\label{lem:clf}
Let $A_*\succ0$ and define
\begin{equation}
V(\tilde g,\eta) := \tfrac{1}{2}\trasp{(\tilde g,\eta)}\,P\,(\tilde g,\eta),
\qquad
P:=\begin{pmatrix} I_k & -\tfrac{1}{2} I_k\\[2pt] -\tfrac{1}{2} I_k & A_*+\tfrac{1}{2} I_k\end{pmatrix}.
\label{eq:P}
\end{equation}
Then $P\succ0$, and for \emph{every} projector $E_{\mathcal{S}}$
($\mathcal{S}\subseteq\{1,\dots,k\}$), the matrix
$M_\mathcal{S}:=\left(\begin{smallmatrix}0&-(A_*+E_{\mathcal{S}})\\ I&-I\end{smallmatrix}\right)$ satisfies
\begin{equation}
\trasp{M_\mathcal{S}}P+P M_{\mathcal{S}}
= \begin{pmatrix} -I & I-E_{\mathcal{S}}\\ I-E_{\mathcal{S}} & E_{\mathcal{S}}-A_*-I\end{pmatrix} =: D_\mathcal{S} \ \prec\ 0 .
\label{eq:dissip-block}
\end{equation}
Consequently the dissipation margin
\begin{equation}
\kappa_0 := \kappa_0(A_*) := \frac{1}{2}\min_{\mathcal{S}\subseteq\{1,\dots,k\}}\lambda_{\min}(-D_{\mathcal{S}})>0
\label{eq:kappa0}
\end{equation}
is uniform over all active sets, and $\diff V/\diff\tau\leq-\kappa_0\Big(\norm{\tilde g}^2+\norm\eta^2\Big)$
along $\diff\tilde g/\diff\tau=-(A_*+E_{\mathcal{S}})\eta$, $\diff\eta/\diff\tau=\tilde g-\eta$. Moreover $\nabla_{\tilde g}V=\tilde g-\tfrac{1}{2}\eta$, so $\norm{\nabla_{\tilde g}V}\leq\tfrac{\sqrt5}{2}\norm{(\tilde g,\eta)}$.
\end{lemma}
The proof of Lemma~\ref{lem:clf} is given at the end of this subsection.

With $\kappa_0$ defined, we state the small-variation condition that replaces Assumption~\Aref{4} in the inequality setting. It is the exact analog of the equality case, with the dissipation margin $\kappa_0$ playing the role of the constant $1$.
\begin{itemize}
\item[(B4)]\phantomsection\label{assm:B4} \emph{Small variation:} there exists a reference matrix $A_*\succ0$ such that along the trajectory
\begin{equation}
\rho := \frac{\sqrt5}{2}\,\sup_t\norm{\jac h(x_t)\trasp{\jac h(x_t)} - A_*}_2 < \kappa_0(A_*).
\label{eq:ineq-smallvar}
\end{equation}
\end{itemize}
As in the equality case, \Bref{4} holds automatically with $\rho=0$ for affine constraints $h(x)=Hx+d$: then $\jac h\trasp{\jac h}=HH^\top$ is constant, and the choice $A_*=HH^\top$ gives $\rho=0<\kappa_0$.

\begin{proof}[Proof of Theorem~\ref{thm:ineq}]
The proof follows the same four steps as the equality-constrained case (Theorem~\ref{thm:p2}). As in the equality case, the nonlinear constraint is handled directly by working with the reference matrix $A_*$ of the small-variation condition~\eqref{eq:ineq-smallvar}. 
Mirroring the abbreviations $A(x)=\jac g\trasp{\jac g}$, $b(x,t)=\jac g\,v_\theta$ of the equality case, we set
    \begin{equation}
A(x) := \jac{h}(x)\trasp{\jac{h}(x)},
\qquad
b(x,t) := \jac{h}(x)\,v_\theta(x,t),
\qquad
r := \relu\Big({-}h(x)-\tfrac{1}{c}\lambda\Big),
\label{eq:iAbr-def}
\end{equation}
so that the slack equation~\eqref{eq:ineq-s} reads $\dot s = c(r-s)$.

Write $\tilde g(x,s) := h(x) + s$. By Lemma~\ref{lem:posinv} and $s_0 \geq 0$ we have $s_t \geq 0$ for all $t$, hence $\norm{\relu(h(x_t))} \leq \norm{\tilde g(x_t,s_t)}$ and it suffices to prove $\norm{\tilde g(x_t,s_t)} \to 0$ as $t \to 1^-$. By Assumptions~\Bref{1}--\Bref{3} we may set $G:=\sup_t\norm{h(x_t)}<\infty$, $S:=\sup_t\norm{s_t}<\infty$, and $\norm{b(x_t,t)} \leq LB$.

\medskip
\noindent\textbf{Step 1 (Exponential time change).}
Let $\tau:=-\log(1-t)$, so $1-t=e^{-\tau}$, and $\eta:=(1-t)\lambda=e^{-\tau}\lambda$.
Applying the chain rule to $\tilde g=h+s$, $\dot{\tilde g}=\jac h(x)\dot x+\dot s$;
substituting \eqref{eq:ineq-x}, \eqref{eq:ineq-s} and converting to
$\tau$-time via $\diff\tilde g/\diff\tau=e^{-\tau}\dot{\tilde g}$ and
$\lambda=e^\tau\eta$,
\begin{equation}
\frac{\diff\tilde g}{\diff\tau}
= -A(x)\,\eta + e^{-\tau}\Big[b(x,t) - c\,A(x)\,\tilde g\Big] + c\,e^{-\tau}\,(r-s),
\label{eq:tau}
\end{equation}
For the dual variable, $\dot\lambda=\tilde g/(1-t)^2$ gives, exactly as in
the equality case,
\begin{equation}
\frac{\diff\eta}{\diff\tau} = \tilde g - \eta,
\qquad\text{and}\qquad
\frac{\diff s}{\diff\tau} = e^{-\tau}\dot s = c\,e^{-\tau}\,(r-s).
\label{eq:ieta-s-tau}
\end{equation}

The ReLU acts componentwise; let $\mathcal{S}(\tau):=\setdef{i}{{-}h^i-\tfrac1c\lambda^i>0}$
be the active set (the indices where $\relu$ is identity) and $E_{\mathcal{S}}:=\mathrm{diag}(\mathbf 1_{i\in\mathcal{S}})$
the corresponding $0/1$ diagonal matrix. On active components
$r^i=-h^i-\tfrac{1}{c}\lambda^i$; substituting (using $\lambda=e^\tau\eta$,
$h+s=\tilde g$) the term $c\,e^{-\tau}r$ contributes $-\eta_i$ to
$\diff\tilde g_i/\diff\tau$ beyond $-A(x)\eta$, while inactive components
contribute $r^i=0$. Hence
\begin{equation}
\frac{\diff\tilde g}{\diff\tau} = -\Big(A(x)+E_{\mathcal{S}}\Big)\eta + e^{-\tau}\Big[b - c(A(x)+E_{\mathcal{S}})\tilde g\Big] - c\,e^{-\tau}s_{\mathcal{S}^c},
\label{eq:regime}
\end{equation}
where $\mathcal{S}^c$ is the complement of $\mathcal{S}$ and $s_{\mathcal{S}^c} = E_{\mathcal{S}^c} s$. As $\tau\to\infty$ the factor $e^{-\tau}$ vanishes, and~\eqref{eq:regime} together with $\diff\eta/\diff\tau=\tilde g-\eta$ approaches, in each active-set regime $\mathcal{S}$, the switching system $\diff\tilde g/\diff\tau=-(A(x)+E_{\mathcal{S}})\eta$, $\diff\eta/\diff\tau=\tilde g-\eta$. Freezing $A(x)$ at a fixed reference $A_*$ renders this autonomous in each active set. Because the active set $\mathcal{S}$ may change along the trajectory, this is a linear switching system and this Lyapunov function must dissipate simultaneously in every regime.

\noindent\textbf{Step 2 (Lyapunov function).}
Let $A_*\succ0$ be the reference matrix of the small-variation condition~\eqref{eq:ineq-smallvar}, and take the Lyapunov function $V$ of Lemma~\ref{lem:clf} built from $A_*$. Write $z := (\tilde g,\eta)$ so that $\norm z^2 = \norm{\tilde g}^2 + \norm\eta^2$. By Lemma~\ref{lem:clf}, $P\succ0$, so $V$ is positive definite and, with $c_1:=\tfrac{1}{2}\lambda_{\min}(P)$ and $c_2:=\tfrac{1}{2}\lambda_{\max}(P)$,
\begin{equation}
c_1\norm z^2 \leq V(\tilde g,\eta) \leq c_2\norm z^2.
\label{eq:iV-bounds}
\end{equation}

\emph{Dissipation along the autonomous limit.} Fix an active set $\mathcal{S}$. Along the autonomous system $\diff\tilde g/\diff\tau=-(A_*+E_{\mathcal{S}})\eta$, $\diff\eta/\diff\tau=\tilde g-\eta$ obtained by freezing $A(x)$ at $A_*$ and dropping the $e^{-\tau}$ forcing in~\eqref{eq:regime}, Lemma~\ref{lem:clf} gives
\begin{equation}
\frac{\diff V}{\diff\tau}\Big|_{\mathrm{autonomous}} = \trasp z\,P\,M_{\mathcal{S}}\,z \leq -\kappa_0\norm z^2 ,
\label{eq:iauto-dissip}
\end{equation}
with $M_{\mathcal{S}}$ from Lemma~\ref{lem:clf}, and $\kappa_0 > 0$ independent of $\mathcal{S}$.

\medskip
\noindent\textbf{Step 3 (Decomposition along the true dynamics).}
Fix the active set $\mathcal{S}$ and split the constraint matrix in~\eqref{eq:regime} as $A(x)=A_*+\Delta(x)$, $\Delta(x):=A(x)-A_*$. Differentiating $V$ along~\eqref{eq:regime} and the $\eta$-dynamics~\eqref{eq:ieta-s-tau} and substituting this splitting,
\begin{equation*}
\frac{\diff V}{\diff\tau}
= \underbrace{\trasp z\,P\,M_{\mathcal{S}}z}_{\leq\,-\kappa_0 \norm z^2\ \eqref{eq:iauto-dissip}}
\;\underbrace{-\,(\nabla_{\tilde g}V)^{\!\top}\Delta(x)\,\eta}_{=:\,R_\Delta}
\;+\; (\nabla_{\tilde g}V)^{\!\top}\underbrace{\left(e^{-\tau}\Big[b - c(A(x)+E_{\mathcal{S}})\tilde g\Big] - c\,e^{-\tau}s_{\mathcal{S}^c}\right)}_{\text{forcing}} ,
\end{equation*}
where $\nabla_{\tilde g}V=\tilde g-\tfrac{1}{2}\eta$. As in the equality case, the deviation $\Delta(x)$ enters only through the constraint block, and does so identically in every regime, since it shifts each $A_*+E_{\mathcal{S}}$ by the same amount.

\emph{Bound on $R_\Delta$ (Jacobian variation).} Using $\norm{\nabla_{\tilde g}V}\leq\tfrac{\sqrt5}{2}\norm z$ from Lemma~\ref{lem:clf} and $\norm{\Delta(x)\eta}\leq\norm{\Delta(x)}_2\norm\eta\leq\norm{\Delta(x)}_2\norm z$,
\begin{equation*}
\abs{R_\Delta}\leq \tfrac{\sqrt5}{2}\,\norm{\Delta(x)}_2\norm z^2 \leq \rho\,\norm z^2 ,
\end{equation*}
where $\rho=\tfrac{\sqrt5}{2}\sup_t\norm{\Delta(x_t)}_2$ as in~\eqref{eq:ineq-smallvar}.

\emph{Bound on the forcing.} By \Bref{1}, $\norm b \leq LB$; by \Bref{2}, $\norm{A(x)+E_{\mathcal{S}}}_2\leq a_{\max}+1$; by \Bref{3}, $\norm{\tilde g}\leq \norm h+\norm s\leq G+S=:G'$ and $\norm s\leq S$. Hence the forcing is bounded in norm by $e^{-\tau}B_0$ with $B_0:=LB+c(a_{\max}+1)G'+cS$, and by Cauchy--Schwarz,
\begin{equation*}
(\nabla_{\tilde g}V)^{\!\top}(\text{forcing}) \leq e^{-\tau}\tfrac{\sqrt5}{2}B_0\norm z =: e^{-\tau}K\norm z ,
\qquad K=\tfrac{\sqrt5}{2}\Big(LB+c(a_{\max}+1)G'+cS\Big) .
\end{equation*}

\emph{Combined inequality.} Collecting the three terms,
\begin{equation}\label{eq:ineq-dissip}
\frac{\diff V}{\diff\tau}
\leq -(\kappa_0-\rho)\norm z^2 + e^{-\tau}K\norm z ,
\end{equation}
with $\kappa_0-\rho>0$ by the small-variation condition~\eqref{eq:ineq-smallvar}.

\medskip
\noindent\textbf{Step 4 (Comparison lemma and explicit rate).}
Set $y:=\sqrt V\geq 0$. Where $V>0$, $\frac{\diff y}{\diff\tau}=\frac{1}{2\sqrt V}\frac{\diff V}{\diff\tau}$; bounding the dissipation term in~\eqref{eq:ineq-dissip} from below via $\norm z^2 \ge V/c_2 = y^2/c_2$ and the forcing term from above via $\norm z \leq \sqrt{V/c_1} = y/\sqrt{c_1}$ (both from~\eqref{eq:iV-bounds}),
\begin{equation*}
\frac{\diff y}{\diff\tau}
\leq \frac{1}{2y}\Big[{-}(\kappa_0-\rho)\frac{y^2}{c_2}\Big]
+\frac{1}{2y}\,e^{-\tau}K\frac{y}{\sqrt{c_1}} = -\alpha y + \gamma e^{-\tau},
\end{equation*}
where $\alpha := \tfrac{\kappa_0-\rho}{2c_2} > 0$ and $\gamma := \tfrac{K}{2\sqrt{c_1}} > 0$. We are left with a scalar differential inequality for $y$. The comparison lemma~\cite[Lemma~3.4]{khalil2002nonlinear} implies that for all $\tau \in [0,\infty)$
\begin{equation*}
y(\tau)\leq y(0)e^{-\alpha\tau}+\gamma\!\int_0^\tau e^{-\alpha(\tau-u)}e^{-u}\,\diff u
= y(0)e^{-\alpha\tau}+\gamma
\begin{cases}
\dfrac{e^{-\alpha\tau}-e^{-\tau}}{1-\alpha}, & \alpha\neq1,\\[1.5ex]
\tau\,e^{-\tau}, & \alpha=1.
\end{cases}
\end{equation*}

We thus conclude that $y(\tau) \to 0$.
Moreover, the top-left block of $P$ in Lemma~\ref{lem:clf} is $I_k$, so
$\lambda_{\max}(P)\ge1$, and with $\kappa_0\le\tfrac{1}{2}$ and $\rho\ge0$ we have
$\alpha=(\kappa_0-\rho)/\lambda_{\max}(P)\le\kappa_0\le\tfrac{1}{2}<1$. Therefore, there exists a constant $M > 0$ such that $y(\tau) \leq M\, e^{-\alpha \tau}$ for all $\tau \geq 0$. 
Converting to original time via $e^{-\tau} = 1-t$, and using $\norm{\relu(h(x_t))} \le \norm{\tilde g} \le \norm z \le y/\sqrt{c_1}$,
\begin{equation*}
\norm{\relu(h(x_t))} \le \frac{M}{\sqrt{c_1}}\,(1-t)^\alpha = O\big((1-t)^\alpha\big) \quad \text{and } \quad (1-t)\|\lambda_t\|  = \|\eta\| = O\big((1-t)^\alpha\big),
\end{equation*}
as $t \to 1^-$. In particular $\norm{\relu(h(x_t))} \to 0$, so the generated sample is feasible in the limit.
\end{proof}

We now prove Lemma~\ref{lem:clf}.
\begin{proof}[Proof of Lemma~\ref{lem:clf}]
\emph{Positive definiteness.} The top-left block of $P$ is $I\succ0$, and
its Schur complement is $(A_*+\tfrac{1}{2} I)-(\tfrac{1}{2} I)\trasp{(\tfrac{1}{2} I)}=A_*+\tfrac14 I\succ0$,
so $P\succ0$.

\emph{Dissipation.} A direct block computation gives
$\trasp{M_{\mathcal{S}}}P+P M_{\mathcal{S}}=\left(\begin{smallmatrix}-I & I-E\\ I-E & E-A_*-I\end{smallmatrix}\right)=:D_{\mathcal{S}}$.
The top-left block is $-I\prec0$; the Schur complement of $D_{\mathcal{S}}$ is
\begin{equation*}
(E_{\mathcal{S}}-A_*-I)-(I-E_{\mathcal{S}})(-I)^{-1}(I-E_{\mathcal{S}}) = (E_{\mathcal{S}}-A_*-I)+(I-E_{\mathcal{S}})^2 .
\end{equation*}
Because $E_{\mathcal{S}}$ is a $0/1$ diagonal projector, $E_{\mathcal{S}}^2=E_{\mathcal{S}}$, hence
$(I-E_{\mathcal{S}})^2=I-E_{\mathcal{S}}$, and the Schur complement collapses to
\begin{equation*}
(E_{\mathcal{S}}-A_*-I)+(I-E_{\mathcal{S}}) = -A_* \ \prec\ 0 .
\end{equation*}
Thus $D_{\mathcal{S}}\prec0$ for every $E_{\mathcal{S}}$. Since the active sets $\mathcal{S}$
form a finite set, $\kappa_0=\tfrac{1}{2}\min_{\mathcal{S}}\lambda_{\min}(-D_{\mathcal{S}})>0$. Along the autonomous flow, $\diff V/\diff\tau = \trasp z P M_{\mathcal{S}}z=\tfrac{1}{2}\trasp z D_{\mathcal{S}}z\leq-\tfrac{1}{2}\lambda_{\min}(-D_{\mathcal{S}})\norm z^2\leq-\kappa_0\norm z^2$, which gives the stated bound.

\emph{Gradient bound.} $\nabla_{\tilde g}V=\tilde g-\tfrac{1}{2}\eta$, so by Cauchy-Schwarz, $\norm{\nabla_{\tilde g}V}\leq\norm{\tilde g}+\tfrac{1}{2}\norm\eta\leq\tfrac{\sqrt5}{2}\norm{(\tilde g,\eta)}$.
\end{proof}

\section{Different Choices for $p$} \label{app:different-p}

In this section, we provide some more analysis for the equality-constrained case for $p \neq 2$. To summarize our results,
\begin{itemize}
    \item For $p \in [1,2)$, we show that the residual $\|g(x_t)\|$ converges to a limit as $t \to 1^-$, but this limit is generally nonzero;
    \item $p = 2$ is the case studied in Theorem~\ref{thm:p2}. We prove constraint satisfaction under Assumptions~\Aref{1}-\Aref{4};
    \item For $p > 2$, we do not have general convergence results though analysis of a simplified scalar setting in Appendix~\ref{app:p-greater-2} indicates that we can see constraint satisfaction in this case.
\end{itemize}

To facilitate the presentation, we first present an essential change of coordinates.
For $p \in [1,\infty)$, define
\begin{equation}\label{eq:cov}
\eta(t) := (1-t)^{p-1}\lambda(t), \qquad \tau := -\log(1-t),
\end{equation}
so $\tau \in [0, \infty)$ and $1-t = e^{-\tau}$. Note that when $p = 1$, $\eta(t) = \lambda(t)$.

Using these new coordinates, our dynamics become
\begin{align}
\frac{\diff x}{\diff\tau} &= -e^{(p-2)\tau}\,\trasp{\nabla g(x)}\,\eta + e^{-\tau}F(x, \tau), \label{eq:xtau}\\
\frac{\diff\eta}{\diff\tau} &= g(x) - (p-1)\eta, \label{eq:etatau}
\end{align}
where $F(x, \tau) := v_\theta(x, t(\tau)) - c\,\trasp{\nabla g(x)}\,g(x)$.

\subsection{$p \in [1,2)$}\label{app:p12}

The exponent $p$ enters \eqref{eq:xtau} only through the gain
$e^{(p-2)\tau}$ multiplying the dual correction, and this gain determines
the steady-state behavior of these dynamics. At $p = 2$ the gain is constant, and
\eqref{eq:xtau}-\eqref{eq:etatau} approach an autonomous system as
$\tau \to \infty$ (this is the case studied in Theorem~\ref{thm:p2}). For
$p > 2$ the gain grows without bound. 
For $p \in [1,2)$, by contrast, the gain decays
exponentially in $\tau$. The following proposition makes this
precise. 
 
\begin{proposition}
\label{prop:small-p}
Let $p \in [1,2)$ and suppose Assumptions~\Aref{1}-\Aref{3} hold along the
trajectory of \eqref{eq:ldf}. Then there exists $x_1^\star \in \mathbb{R}^n$
such that $x_t \to x_1^\star$ as $t \to 1^-$, and hence
$g(x_t) \to g_1 := g(x_1^\star)$. The constraint is satisfied in the limit if
and only if $g_1 = 0$, which the dynamics do not enforce.
\end{proposition}
 
\begin{proof}
By Assumptions~\Aref{1} and~\Aref{3}, there are constants $G, L, B \geq 0$ with
$\|g(x_\tau)\| \leq G$, $\|\nabla g(x_\tau)\|_2 \leq L$, and
$\|F(x_\tau,\tau)\| \leq B + cLG =: M_F$ along the trajectory.
 
First we bound $\eta$. Solving the linear equation \eqref{eq:etatau} by
variation of constants, with $\eta(0) = \lambda_0 = 0$,
\begin{equation*}
    \eta(\tau) = \int_0^\tau e^{-(p-1)(\tau - u)}\, g(x(u))\, \diff u,
\end{equation*}
so for $p \in (1,2)$ we get $\|\eta(\tau)\| \leq G/(p-1)$ for all $\tau$,
and for $p = 1$ we get $\|\eta(\tau)\| \leq G\tau$.
 
Substituting these bounds into \eqref{eq:xtau},
\begin{equation*}
    \left\| \frac{\diff x}{\diff \tau} \right\|
    \leq L\, e^{(p-2)\tau} \|\eta(\tau)\| + M_F\, e^{-\tau}
    =: \varphi(\tau),
\end{equation*}
and since $p - 2 < 0$, the function $\varphi$ is integrable on
$[0,\infty)$ in both cases. 

Since $\diff x/\diff\tau$ is bounded and integrable on $[0,\infty)$, the following integral is absolutely convergent which lets us define the limit
\begin{equation}
x_1^\star := x_0 + \int_0^\infty \frac{\diff x}{\diff \tau} \diff\tau.
\end{equation}
We can then see that $\|x(\tau) - x_1^\star\| \leq \int_\tau^\infty \varphi(u) \diff u \to 0$ as $\tau \to \infty$. Therefore we conclude that $x(\tau) \to x_1^\star$ as $\tau \to \infty$ (or equivalently that $x_t \to x_1^\star$ as $t \to 1^-$) and that $g(x_t) \to g(x_1^\star)$ as $t \to 1^{-}$ by continuity. 
\end{proof}

We exhibit a simple instance in which the failure described by
Proposition~\ref{prop:small-p} can be established by a direct argument in
the $\tau$-coordinates. Let $n = m = 1$ with $g(x) = x$, and take
$v_\theta = 0$ and $c = 0$: the flow has no drift and no penalty, so
the dual variable is the only mechanism acting on the sample, and its only
task is to bring $x$ to the constraint set $\{x = 0\}$. With $p = 1$ we
have $\eta = \lambda$, and the general dynamics
\eqref{eq:xtau}-\eqref{eq:etatau} become the simple
system
\begin{equation}
    \frac{\diff x}{\diff \tau} = -e^{-\tau}\,\eta,
    \qquad
    \frac{\diff \eta}{\diff \tau} = x,
    \qquad
    x(0) = x_0, \quad \eta(0) = 0,
    \label{eq:p1-tau}
\end{equation}
where $\tau = \log\frac{1}{1-t} \in [0,\infty)$ and constraint satisfaction
as $t \to 1^-$ corresponds to $x(\tau) \to 0$ as $\tau \to \infty$.
 
Integrating the two equations of \eqref{eq:p1-tau} from $0$ gives the two
identities
\begin{equation}
    \eta(\tau) = \int_0^\tau x(u)\,\diff u,
    \qquad
    x(\tau) = x_0 - \int_0^\tau e^{-u}\,\eta(u)\,\diff u.
    \label{eq:p1-integrals}
\end{equation}
 
\begin{proposition}
\label{prop:p1-tau}
Let $x_0 > 0$. Along \eqref{eq:p1-tau}, $x(\tau) > 0$ and is strictly
decreasing for all $\tau > 0$, and
\begin{equation*}
    x_\infty := \lim_{\tau\to\infty} x(\tau)
    \quad\text{exists with}\quad
    0 < x_\infty \leq \tfrac{1}{2}x_0 .
\end{equation*}
Hence
$x_\infty \neq 0$ whenever $x_0 \neq 0$, and the constraint is not
satisfied in the limit.
\end{proposition}
 
\begin{proof}
We will use the following fact several times:
\begin{equation}
    \int_0^\infty u\,e^{-u}\,\diff u = 1 .
    \label{eq:gamma2}
\end{equation}
 
First we show that when $x_0 > 0$, we have that $x_\tau > 0$ for all $\tau > 0$. To this end,
let $\tau^\star := \sup\setdef{T}{ x(u) > 0 \text{ for all } u \in [0,T) }$,
which is positive since $x(0) = x_0 > 0$ and $x$ is continuous. On
$[0,\tau^\star)$ we have $x > 0$, so by the first identity
in~\eqref{eq:p1-integrals} $\eta > 0$, hence
$\diff x/\diff\tau = -e^{-\tau}\eta < 0$: the residual is strictly
decreasing, and in particular $x(u) < x_0$ for $u \in (0,\tau^\star)$.
Therefore $\eta(u) = \int_0^u x(r)\,\diff r < x_0\,u$ for every
$u \in (0,\tau^\star)$, and the second identity in~\eqref{eq:p1-integrals}
gives, for any $\tau \in (0,\tau^\star)$,
\begin{equation*}
    x(\tau)
    = x_0 - \int_0^\tau e^{-u}\eta(u)\,\diff u
    > x_0 - x_0\int_0^\tau u\,e^{-u}\,\diff u
    > x_0 - x_0\int_0^\infty u\,e^{-u}\,\diff u
    = 0 ,
\end{equation*}
using~\eqref{eq:gamma2} in the last step. Thus $x$ is bounded away from
zero on every compact subinterval and can never reach it, so
$\tau^\star = \infty$: the residual stays positive and strictly decreasing
for all $\tau$.
 
Being positive and strictly decreasing, $x(\tau)$ has a limit
$x_\infty \in [0, x_0)$. Since $x(r) < x_0$ for every $r > 0$, the strict
inequality $\eta(u) < x_0 u$ holds for all $u > 0$. Multiplying by the
positive weight $e^{-u}$ and integrating over $(0,\infty)$,
\begin{equation*}
    \int_0^\infty e^{-u}\eta(u)\,\diff u
    < x_0\int_0^\infty u\,e^{-u}\,\diff u
    = x_0 .
\end{equation*}
Letting $\tau \to \infty$ in the second identity
of~\eqref{eq:p1-integrals} therefore yields
\begin{equation*}
    x_\infty
    = x_0 - \int_0^\infty e^{-u}\eta(u)\,\diff u
    > x_0 - x_0 = 0 .
\end{equation*}
For the upper bound, $x(u) \geq x_\infty$ gives
$\eta(u) \geq x_\infty u$, hence
$\int_0^\infty e^{-u}\eta(u)\,\diff u \ge x_\infty$ by~\eqref{eq:gamma2},
so $x_\infty = x_0 - \int_0^\infty e^{-u}\eta \leq x_0 - x_\infty$, i.e.\
$x_\infty \leq \tfrac{1}{2} x_0$. The case that $x_0 < 0$ is handled by symmetry and is omitted.
\end{proof}

\subsection{$p > 2$}\label{app:p-greater-2}

To provide supporting evidence that we can see constraint satisfaction for $p > 2$, we revisit the scalar example from the previous section ($n = m = 1$, $g(x) = x$, and $v_\theta = 0$). Using the same change of coordinates $\tau = -\log(1-t)$ and $\eta = (1-t)^{p-1}\lambda$, we are left with the coupled system
\begin{equation}
    \frac{\diff x}{\diff\tau} = -e^{(p-2)\tau}\,\eta,
    \qquad
    \frac{\diff\eta}{\diff\tau} = x - (p-1)\eta,
    \qquad x(0) = x_0,\ \eta(0) = 0.
    \label{eq:pg2-sys}
\end{equation}
Differentiating the first equation and substituting the second turns
\eqref{eq:pg2-sys} into a single scalar equation for the residual,
\begin{equation}
    \frac{\diff^2 x}{\diff\tau^2}
    + \frac{\diff x}{\diff\tau}
    + e^{(p-2)\tau}\, x = 0,
    \qquad
    x(0) = x_0,\quad \frac{\diff x}{\diff\tau}(0) = 0,
    \label{eq:pg2-scalar}
\end{equation}
which is a damped oscillator with unit damping and a stiffness
$e^{(p-2)\tau}$ that grows unbounded when $p > 2$. Constraint
satisfaction as $t \to 1^-$ corresponds to $x(\tau) \to 0$. Note that when $p = 2$, this is a standard damped oscillator whose solution trajectories tend to the origin exponentially quickly, agreeing with the results of Theorem~\ref{thm:p2}.

First, we will choose the substitution  $x(\tau) = e^{-\tau/2} y(\tau)$. Computing the derivatives of $x(\tau)$ and substituting them into the equation~\eqref{eq:pg2-scalar}, we get the differential equation for $y$:
\begin{equation}
    \frac{\diff^2 y}{\diff \tau^2} + \left(e^{(p-2)\tau} - \frac{1}{4}\right) y = 0.
\end{equation}

Let $Q(\tau) = e^{(p-2)\tau} - \frac{1}{4}$. Since $p > 2$, we see that $Q$ is strictly positive and non-decreasing, i.e., $Q(\tau) > 0$ and $Q'(\tau) \geq 0$. Going forward, we will use $'$ to denote a derivative with respect to $\tau$. 

Define the energy function
\begin{equation}
    V(\tau) = \frac{1}{Q(\tau)}(y')^2 + y^2.
\end{equation}
Since $Q(\tau) > 0$, $V(\tau)$ is always nonnegative and satisfies that $y^2 \leq V(\tau)$, meaning that if we can show that $V$ is bounded, we will conclude that $y$ is as well. Taking a $\tau$-derivative of $V$ yields
\begin{align*}
    \frac{\diff V}{\diff \tau} &= \frac{2y'y''Q(\tau) -(y')^2 Q'(\tau)}{Q(\tau)^2} + 2yy' = \frac{-2yy'Q(\tau)^2 - (y')^2Q'(\tau)}{Q(\tau)^2} + 2yy' \\ &= -\frac{Q'(\tau)}{Q(\tau)^2} (y')^2 \leq 0.
\end{align*}
Therefore, for all $\tau \in [0,\infty)$, we have that $V(\tau) \leq V(0)$. Since $y^2(\tau) \leq V(\tau)$, we conclude that $y$ is bounded and cannot grow to infinity. Mapping back to $x$, $|x(\tau)| = e^{-\tau/2}|y(\tau)| \to 0$ as $\tau \to \infty$. Therefore, we get constraint satisfaction in this case for any initial condition and any $p > 2$.

\end{document}